\documentclass[11pt,reqno]{amsart}
\usepackage{amsmath,amssymb,amsthm,amsfonts,mathrsfs,latexsym,times,color,bm}
\usepackage{amsmath,amsfonts}
\usepackage{amsthm}
\usepackage{graphicx}
\usepackage{color}
\usepackage{multicol}
\usepackage{hyperref}

\newtheorem{thm}{Theorem}[section]

\newtheorem{Lemma}[thm]{Lemma}
\newtheorem{remark}{Remark}[section]
\newtheorem{theorem}[thm]{Theorem}
\newtheorem{proposition}[thm]{Proposition}

\numberwithin{equation}{section}

\newcommand{\beq}{\begin{equation}}
\newcommand{\eeq}{\end{equation}}
\newcommand{\ben}{\begin{eqnarray}}
\newcommand{\een}{\end{eqnarray}}
\newcommand{\beno}{\begin{eqnarray*}}
\newcommand{\eeno}{\end{eqnarray*}}

\newcommand{\bv}{\mathbf{v}}

\newcommand{\bh}{\mathbf{h}}

\newcommand{\bF}{\mathbf{F}}

\newcommand{\bG}{\mathbf{G}}
\newcommand{\bE}{\mathbf{E}}

\newcommand{\bH}{\mathbf{H}}

\newcommand{\by}{\mathbf{y}}
\newcommand{\bx}{\mathbf{x}}
\newcommand{\bU}{\mathbf{U}}

\numberwithin{equation}{section}

\begin{document}
\title[Incompressible Elastic equations]{Local well-posedness for incompressible neo-Hookean Elastic equations in almost critical Sobolev spaces}


\author{Huali Zhang}
\address{Hunan University, School of Mathematics, Lushan South Road in Yuelu District, Changsha, 410882, People's Republic of China.}

\email{hualizhang@hnu.edu.cn}

\date{\today}


\begin{abstract}
Inspired by a pioneer work of Andersson-Kapitanski \cite{AK}, we prove the local well-posedness of the Cauchy problem of incompressible neo-Hookean equations if the initial deformation and velocity belong to $H^{s+1}(\mathbb{R}^n) \times H^{s}(\mathbb{R}^n), s>\frac{n+1}{2}$ ($n=2,3$). Moreover, if the initial data is small, then we can lower the regularity to $s>\frac{n}{2}$, where $\frac{n+2}{2}$ and $\frac{n}{2}$ is respectively a scaling-invariant exponent for deformation and velocity in Sobolev spaces. Our new observation relies on two folds: a reduction to a second-order wave-elliptic system of deformation and velocity; and a "wave-map type" null form intrinsic in this coupled system. In particular, the wave nature with "wave-map type" null form allows us to prove a bilinear estimate of Klainerman-Machedon type for nonlinear terms. So we can lower $\frac12$-order regularity in 3D and $\frac34$-order regularity in 2D for well-posedness compared with \cite{AK}.
\end{abstract}
	
\maketitle
\section{Introduction}
This paper considers solutions to the Cauchy problem of incompressible neo-Hookean elastic equations for low regularity initial data in two or three spatial dimensions. To fix the formulation of this system, let us introduce the motion of the neo-Hookean material, which can also be found in \cite{AK}.
\subsection{The motion in Lagrangian settings}
We assume that an elastic body occupies the whole space $\mathbb{R}^n$ with coordinates $\by$, and $\by=(y^1,y^2,\cdots, y^n)^{\text{T}}$. A generic point $\by\in \mathbb{R}^n$ will find itself at the location $\bx(t,\by)$ at time $t$, and $\bx=(x^1, x^2,\cdots, x^n)^{\text{T}}$. Incompressibility means that
\begin{equation}\label{in0}
	\text{det} \frac{\partial \bx(t,\by)}{\partial \by}=1.
\end{equation}
The action functional
\begin{equation*}
	\int \int \frac12 \left| \frac{\partial \bx(t,\by)}{\partial t} \right|^2+ p(t,\by)\left( \text{det} \frac{\partial \bx(t,\by)}{\partial \by}- 1 \right) d\by dt,
\end{equation*}
gives the Euler-Lagrange equations
\begin{equation}\label{in40}
	\frac{\partial^2 x^i}{\partial t^2}+ \frac{\partial p}{\partial x^i}=0, \quad i=1,2,\cdots,n.
\end{equation}
Above, $p(t,\by)$ is a Lagrangian multiplier. Equation \eqref{in40} describes the ideal fluid (incompressible Euler equations), cf. \cite{Lag}. For homogeneous, isotropic and hyperelastic materials, adding the potential energy, the action takes form
\begin{equation}\label{in3}
	\int \int \frac12 \left| \frac{\partial \bx(t,\by)}{\partial t} \right|^2- W\left( \frac{\partial \bx(t,\by)}{\partial \by} \right)+ p(t,\by)\left( \text{det} \frac{\partial \bx(t,\by)}{\partial \by}- 1 \right) d\by dt,
\end{equation}
Here $W$ is the stored energy functions and
$p$ is a Lagrangian multiplier, cf. \cite{MH}. Let us focus on the most simplest case, i.e., the neo-Hookean case, in which the strain
energy functional is simply given by
\begin{equation*}
	W\left( \frac{\partial \bx(t,\by)}{\partial \by} \right)=\frac12 \left| \frac{\partial \bx}{\partial \by} \right|^2=\frac12  \sum^n_{i,a=1} \left| \frac{\partial x^i}{\partial y^a} \right|^2.
\end{equation*}
Clearly, in the neo-Hookean case, the corresponding Euler-Lagrange equation of \eqref{in3} are\footnote{
In the summation the repeated indices run from $1$ to $n$.}
\begin{equation}\label{in4}
		 \frac{\partial^2 x^i}{\partial t^2}- \frac{\partial^2 x^i}{\partial y^a \partial y^a}+ \frac{\partial p}{\partial x^i}=0, \quad i=1,2,\cdots,n.
\end{equation}
Therefore, Equations \eqref{in4} complemented by the constraint \eqref{in0} describes the motion of the incompressible neo-Hookean material in the Lagrange cooridnates $(t,\by)$:
\begin{equation}\label{in5}
	\begin{cases}
	& \frac{\partial^2 x^i}{\partial t^2}- \frac{\partial^2 x^i}{\partial y^a \partial y^a}+ \frac{\partial p}{\partial x^i}=0, \quad i=1,2,\cdots,n,
	\\
	& \text{det} (\frac{\partial \bx}{\partial \by} )=1.
	\end{cases}
\end{equation}
To be simple, we consider\footnote{Our result also works for $\bx=A\by+\bU(t,\by)$, where $A$ is a constant matrix.}
\begin{equation}\label{aa0}
	\bx=\by+\bU(t,\by),
\end{equation}
where $\bU=(U^1,U^2,\cdots,U^n)$. By substituting \eqref{aa0} to \eqref{in5}, then \eqref{in5} becomes to
\begin{equation}\label{aa1}
	\begin{cases}
		& \frac{\partial^2 U^i}{\partial t^2}- \frac{\partial^2 U^i}{\partial y^a \partial y^a}+ \frac{\partial p}{\partial x^i}=0, \quad i=1,2,\cdots,n,
		\\
		& \text{det} (I+\frac{\partial \bU}{\partial \by} )=1.
	\end{cases}
\end{equation}
For we study the Cauchy problem of \eqref{aa1}, so we set the initial conditions
\begin{equation}\label{aa2}
		( \bU, \frac{\partial \bU}{\partial t})|_{t=0}=(\bU_0, \bv_0) \in H^{s+1}(\mathbb{R}^n) \times H^{s}(\mathbb{R}^n) .
\end{equation}
Seeing \eqref{aa1}, it is invariant under the scaling
\begin{equation}\label{aa3}
	(t,\by)\mapsto (\lambda t, \lambda \by).
\end{equation}
This tells us that $\bU_{\lambda}=\lambda^{-1}\bU(\lambda t, \lambda \by)$ is also a solution  of \eqref{aa1} if $\bU$ is a solution of \eqref{aa1}. Under the scaling, the homogeneous Sobolev norm of the initial data scales as
\begin{equation}\label{aa4}
	\begin{split}
		& \| \bU_{\lambda}(0,\cdot)\|_{\dot{H}^{s+1}(\mathbb{R}^n)}=\lambda^{s-\frac{n}{2}} \| \bU(0,\cdot) \|_{\dot{H}^{s+1}(\mathbb{R}^n)},
		\\
		& \| \frac{\partial}{\partial t}\bU_{\lambda}(0,\cdot)\|_{\dot{H}^s(\mathbb{R}^n)}=\lambda^{s-\frac{n}{2}} \| \frac{\partial}{\partial t}\bU(0,\cdot) \|_{\dot{H}^s(\mathbb{R}^n)}.
	\end{split}
\end{equation}
Taking $s_c=\frac{n}{2}$, then the norm $\| \bU_0\|_{ \dot{H}^{s_c+1} }$ and $\| \bv_0\|_{ \dot{H}^{s_c} }$ norm of the initial data $(\bU_0,\bv_0)$ is preserved under the scaling. Hence, $s_c+1$ is called the critical exponent of the deformation, and $s_c$ the critical exponent of the velocity. In the viewpoint of scaling, one would expect local well-posedness for data in the Sobolev space ${H}^{s+1} \times {H}^s$ for $s>s_c$ (subcritical case), global existence for small data in $\dot{H}^{s+1} \times \dot{H}^s$ for $s=s_c$ (critical case), and some form of ill-posedness for data in ${H}^{s+1} \times {H}^s$ for $s<s_c$ (supcritical case). In the paper, we study the local well-posedness of \eqref{aa1}-\eqref{aa2} in the meaning of subcritical case. To understand the system \eqref{aa1} much more clearly, let us also reduce it in the Euler picture.
\subsection{The dynamics in Euler settings}
To distinguish between the derivatives in Lagrange coordinate and Euler picture, we denote by $\partial_t$ and $\partial_i$ the partial derivatives in the Euler coordinate and by $\frac{\partial}{\partial t}$ and $\frac{\partial}{\partial y^a}$ the derivatives in the Lagrangian coordinate.

In the Euler picture, the independent variables are $(t,\bx)$. Then the vector
\begin{equation}\label{in6}
	\bar{\bv}(t,\bx)=\frac{\partial \bx}{\partial t},
\end{equation}
is the velocity in the Eulerian coordinates, and we set $\bar{\bv}=(\bar{v}^1,\bar{v}^2,\cdots,\bar{v}^n)$. The incompressibility condition \eqref{in0} translates into
\begin{equation}\label{in7}
	\text{div}\bar{\bv} = \partial_i \bar{v}^i=0.
\end{equation}
To describe the Eulerian form of equation \eqref{in4} for the incompressible neo-Hookean material. We introduce the deformation gradient $\bar{\bF}$ by
\begin{equation}\label{in8}
	\bar{F}^{ia}(t,\bx)=\frac{\partial x^i}{\partial y^a}, \quad a=1,2,\cdots, n.
\end{equation}
We also define the deformation matrix as
\begin{equation}\label{in9}
	\bar{\bF}=( \bar{F}^{ia})_{n\times n}.
\end{equation}
Additionally, from the Piola's identity(cf. \cite{MH}), we have
\begin{equation}\label{in10}
\text{div}\bar{\bF} =0.
\end{equation}
In the Euler picture, \eqref{in4} is transformed to
\begin{equation}\label{in11}
	\partial_t \bar{v}^i + \bar{v}^j \partial_j \bar{v}^i- \bar{F}^{kb}\partial_k \bar{F}^{ib}+\partial_i \bar{p}=0.
\end{equation}
Above, $\bar{p}$ is transformed by $p$ in Euler picture, and $\bar{p}=\bar{p}(t,\bx)=\bar{p}(t,\bx(t,\by))=p(t,\by)$.

On the other hand, we also have
\begin{equation}\label{in17}
	\frac{d}{dt}\bar{F}^{ia} = \partial_t \bar{F}^{ia}(t,\bx(t,\by))= \partial_t \bar{F}^{ia}+ \bar{v}^j \partial_j \bar{F}^{ia},
\end{equation}
and
\begin{equation}\label{in18}
	\frac{d}{dt} \bar{F}^{ia} = \frac{\partial}{\partial_t}  ( \frac{\partial x^i }{\partial y^a})= \frac{\partial}{\partial y^a} ( \frac{\partial x^i }{\partial t})=\frac{\partial {v}^i}{\partial y^a} =\bar{F}^{ka}\partial_k \bar{v}^{i}
\end{equation}
Due to \eqref{in17} and \eqref{in18}, it yields
\begin{equation}\label{in19}
	\partial_t \bar{F}^{ia} + \bar{v}^j \partial_j \bar{F}^{ia}- \bar{F}^{ka}\partial_k \bar{v}^{i}=0.
\end{equation}
Combining \eqref{in7}, \eqref{in10}, \eqref{in11}, and \eqref{in19}, the incompressible neo-Hookean equations in Euler coordinates is described by
\begin{equation}\label{IEE}
\begin{cases}
& \partial_t \bar{v}^i + \bar{v}^j \partial_j \bar{v}^i- \bar{F}^{kb}\partial_k \bar{F}^{ib}+\partial_i \bar{p}=0,
\\
& \partial_t \bar{F}^{ia} + \bar{v}^j \partial_j \bar{F}^{ia}- \bar{F}^{ka}\partial_k \bar{v}^{i}=0,
\\
& \text{div}\bar{\bv}=\text{div}\bar{\bF}=0.
\end{cases}
\end{equation}
Similarly, in the Euler picture, the incompressible Euler equation \eqref{in40} becomes to
\begin{equation}\label{incom}
	\begin{cases}
		& \partial_t \bar{v}^i + \bar{v}^j \partial_j \bar{v}^i+\partial_i \bar{p}=0,
		\\
		& \text{div}\bar{\bv}=0.
	\end{cases}
\end{equation}
As we known, System \eqref{incom} is a fundamental equation in fluid dynamics. 
\subsection{Previous results}
Let us first recall some historical results on incompressible Euler equations \ref{incom}. Kato and Ponce in \cite{KP} proved the local well-posedness of solutions if $\bv_0 \in W^{s,p}(\mathbb{R}^n)$, $s>1+\frac{n}{p}$. Chae in \cite{Chae} proved the local existence of solutions by setting $\bv_0$ in Triebel-Lizorkin spaces. In the opposite direction, Bourgain and Li \cite{BL1, BL2} proved that the Cauchy problem is ill-posed for $\bv_0 \in W^{1+\frac{n}{p}}(\mathbb{R}^n), 1\leq p< \infty, n=2,3$.

Compared with incompressible Euler equations, the well-posedness results for incompressible elastic equations are so different, for there are internal forces in fluids. For small solutions, the 3D global well-posedness of smooth solutions was solved by Sideris-Thomases \cite{Sid}, where the inherent null structure plays an essential role. The almost global behavior of the 2D problem
is studied by Lei-Sideris-Zhou \cite{LSZ}. Short afterward, Lei in \cite{Lei} proved the global existence of solutions for 2D problems in Lagrangian settings. By using a different approach, Wang \cite{Wang} gave an alternative proof of 2D result \cite{Lei} in the Eulerian formulation. For local solutions, Andersson and Kapitanski \cite{AK} initially established the local well-posedness of rough solutions for \eqref{aa1}-\eqref{aa2} if the initial data $(\bU_0, \bv_0 ) \in H^{3+}(\mathbb{R}^3)\times H^{2+}(\mathbb{R}^3)$ or $(\bU_0, \bv_0 ) \in H^{\frac{11}{4}+}(\mathbb{R}^2)\times H^{\frac{7}{4}+}(\mathbb{R}^2)$, with some additional regularity conditions on the vorticity. In \cite{AK}, the key point relies on decomposing vorticities into a first-order half-wave system and establishing a Strichartz estimate of vorticities. Concerning wave equations, even for linear waves, the Strichartz estimates hold only when the regularity of the initial data is greater than $2$ in 3D or greater than $\frac74$ in 2D, please refer to the well-known endpoint result by Keel and Tao \cite{KeelTao}. So if we want to lower the regularity in \cite{AK}, we shall explore other good structures and estimates in the nonlinear terms. In the historical results \cite{Sid, Lei, Wang}, we already know that there exists a null structure inside the nonlinearity. In general, there are three basic types of null forms\footnote{In our paper, we also call $Q_0$-type as wave-map type.}(cf. Klainerman \cite{Klai}): 
\begin{equation}\label{nuform}
	\begin{split}
		& Q_0(f,g)= \partial_t f \partial_t g- \partial_i f  \partial_i g, 
		\\
		& Q_{ij}(f,g)= \partial_i f \partial_j g- \partial_j f  \partial_i g, 
		\\
		& Q_{0j}(f,g)= \partial_t f \partial_j g- \partial_j f  \partial_t g.
	\end{split}
\end{equation}
As known to us, the local well-posedness results depend on the structure of null forms. To explain it, let us introduce the systems of nonlinear wave equations
\begin{equation}\label{nuf0}
	\begin{cases}
		& \square u^{I}=\sum_{J,K}\Gamma^{I}_{JK} B^I_{JK}(\partial u^J,\partial u^K),
		\\
		&(u^I, \partial_t u^I)|_{t=0}=(f^I,g^I)\in H^s(\mathbb{R}^n) \times H^{s-1}(\mathbb{R}^n),
	\end{cases}
\end{equation}
where $u^I: \mathbb{R}^{1+n} \rightarrow \mathbb{R}^n$, $B^I_{JK}$ is a quadratic form, and $\Gamma^{I}_{JK}$ is a smooth function of $u^I$. For $B^I_{JK}$ with the general quadratic forms, the local well-posedness of \eqref{nuf0} for $s>2$ in three dimensions was proved by Ponce-Sideris \cite{PS}, which is sharp in light of the counterexamples of Lindblad
\cite{Lind1,Lind2}. If $n\geq 5$, the local well-posedness of \eqref{nuf0} for $s>\frac{n+5}{4}$ was proved by Tataru \cite{T2}. If $n=2,4$, the local well-posedness of \eqref{nuf0} for $s>\max\{\frac{n}{2}, \frac{n+5}{4}\}$ was obtained by Zhou \cite{Zhou1}. The results for $n\leq 5$ also turn out to be sharp due to the ill-posedness result from Fang-Wang \cite{FW}, Liu-Wang \cite{LW} and Ohlmann \cite{O}. For $B^I_{JK}$ are all null forms, we call \eqref{nuf0} the null-form quadratic wave equations. For $B^I_{JK}$ are only $Q_0$-type null form, we call it $Q_0$-type quadratic wave equations. In these situations, one can expect a better result than in a general quadratic form. For $n=3$, Klainerman and Machedon \cite{KM1} firstly studied the bilinear estimates of null forms. Later, Klainerman and Machedon \cite{KM2} also proved the local well-posedness if $s>\frac32$ for the null-form quadratic wave equations \eqref{nuf0}. If $n\geq 2$, the local well-posedness for $s>\frac{n}{2}$ of $Q_0$-type quadratic wave equations \eqref{nuf0} was established by Klainerman and Selberg \cite{KS1} (see also in Klainerman-Selberg \cite{KS2} for a simplified proof). But, when $n=2$, it is known to be false for the other null forms, for which the best result is for $s>\frac54$ by
Zhou \cite{Zhou}, who also proved that it is sharp. If $n=1$, Keel and Tao \cite{KTao} showed that the wave maps is locally well-posed\footnote{The wave maps has $Q_0$-type null form, which is also a special model of \eqref{nuf0}.} for $s>\frac12$. We should also mention some other important low regularity results on \eqref{nuf0} due to Foschi-Klainerman\cite{FK}, Grigoryan-Nahmod\cite{GN}, Gr\"unrock \cite{Grun}, Selberg \cite{Selberg}, Tao \cite{Tao1,Tao2}, Tataru \cite{T1}, Wang-Zhou \cite{WZ} and so on. 
Let us go back to incompressible elastic equations. Although the results \cite{Sid, Lei, Wang} told us that the system \eqref{aa1} has a null structure, they didn't directly point out the $Q_0$-type of \eqref{aa1}. As we mentioned before, the well-posedness result of \eqref{aa1} may be different if it's $Q_0$-type or $Q_{ij}$, $Q_{0i}$-type. The above important results motivate us to explore the structure of \eqref{aa1} and lower the regularity conditions in \cite{AK}.

Motivated by the insightful work \cite{AK}, in this paper, we present that there is a "$Q_0$-type" null structure for incompressible elastic equations \eqref{in5}. Using the stress term and obtaining the "$Q_0$-type" null structure are our new findings. Although it involves nonlocal operators, by an equivalent transform of norms in Lagrangian coordinates and Euler pictures, we can also prove a bilinear estimate for nonlinear terms in $X_{s,b}$ spaces. So we prove the local well-posedness of \eqref{aa1}-\eqref{aa2} for all $s$ above scaling, $s>\frac{n}{2}$.

Let us state our results as follows.

\subsection{Statement of the result}
Before state our results, let us set 
\begin{equation}\label{sn}
	s_0(n)=
	\begin{cases}
		\frac74, & \ \ n=2,
		\\  2, & \ \ n=3.
	\end{cases}
\end{equation} 
Our first result is about the local well-posedness of incompressible elastic equations.
\begin{theorem}\label{thm}
	Let $n=2,3$ and $\frac{n+1}{2}<s\leq s_0(n)$, where $s_0(n)$ is stated in \eqref{sn}. Then there exists a positive number $T$ (depending on $s$ and $\|(\bU_0,\bv_0)\|_{H^{s+1}\times H^s}$) such that the Cauchy problem of \eqref{aa1}-\eqref{aa2} is locally well-posed on $[0,T]\times \mathbb{R}^n$ for data $(\bU_0,\bv_0)$ in $ H^{s+1}(\mathbb{R}^n) \times H^s(\mathbb{R}^n)$. Moreover, the following energy estimate hold:
	\begin{equation*}
		\begin{split}
			& \|{\bU}\|_{L^\infty_{[0,T]}H^{s+1}}+\|\frac{\partial \bU}{\partial t}\|_{L^\infty_{[0,T]}H^{s}}  \leq C ( \|\bU_0\|_{H^{s+1}} + \|\bv_0\|_{H^s}).
		\end{split}	
	\end{equation*}
\end{theorem}
\begin{remark}
	Compared with Andersson-Kapitanski \cite{AK}, we mainly find a $Q_0$ null form of \eqref{aa1}-\eqref{aa2} and use bilinear estimates to lower the regularity on initial data in Theorem \ref{thm}.
\end{remark}
Our second result is about the local well-posedness of small solutions of incompressible elastic equations.
\begin{theorem}\label{thmb}
	Let $n=2,3$ and $s>\frac{n}{2}$. For sufficiently small positive number $\eta>0$, set the initial data satisfying
	\begin{equation*}
		\|\bU_0\|_{H^{s+1}} + \|\bv_0\|_{H^s} \leq \eta.
	\end{equation*}
	Then the Cauchy problem of \eqref{aa1}-\eqref{aa2} is locally well-posed on $[0,1]\times \mathbb{R}^n$. Moreover, the following energy estimate hold:
	\begin{equation*}
		\begin{split}
			& \|{\bU}\|_{L^\infty_{[0,T]}H^{s+1}}+\|\frac{\partial \bU}{\partial t}\|_{L^\infty_{[0,T]}H^{s}}   \lesssim   \eta.
		\end{split}	
	\end{equation*}
\end{theorem}
\begin{remark}
In Theorem \ref{thmb}, the regularity of initial data reaches nearly the critical exponent $\frac{n}{2}$.
\end{remark}
Our third result is about Strichartz estimates of incompressible elastic equations.
\begin{theorem}\label{thm2}
	Let $n=2,3$ and $s>s_0(n)$, where $s_0(n)$ is stated in \eqref{sn}. Then there exists a positive number $T_*$ (depending on $s$ and $\|(\bU_0,\bv_0)\|_{H^{s+1}\times H^s}$) such that the Cauchy problem of \eqref{aa1}-\eqref{aa2} is locally well-posed on $[0,T_*]\times \mathbb{R}^n$ for data $(\bU_0,\bv_0)$ in the space $ H^{s+1}(\mathbb{R}^n) \times H^s(\mathbb{R}^n)$. Moreover, the following energy estimates 
		\begin{equation*}
		\begin{split}
			& \|{\bU}\|_{L^\infty_{[0,T_*]}H^{s+1}}+\|\frac{\partial \bU}{\partial t}\|_{L^\infty_{[0,T_*]}H^{s}}  \leq C ( \|\bU_0\|_{H^{s+1}} + \|\bv_0\|_{H^s}),
		\end{split}	
	\end{equation*}
	and Strichartz estimates
	\begin{equation}\label{thm20}
	\begin{split}
	& \| d {\bF}, d {\bv}\|_{L^4_{[0,T_*]}L^\infty_x} \leq C ( \|\bU_0\|_{H^{s+1}} + \|\bv_0\|_{H^s}), \quad n=2,
	\\
	& \| d {\bF}, d {\bv}\|_{L^2_{[0,T_*]}L^\infty_x} \leq C ( \|\bU_0\|_{H^{s+1}} + \|\bv_0\|_{H^s}), \quad n=3,
	\end{split}	
	\end{equation}
	hold. Above, the operator $d=(\frac{\partial}{\partial t}, \frac{\partial}{\partial y^1},\frac{\partial}{\partial y^2},\cdots,\frac{\partial}{\partial y^n})^{\mathrm{T}}$.
\end{theorem}
\begin{remark}
	In Theorem \ref{thm2}, we prove the Strichartz estimate of solutions for Cauchy problem \eqref{aa1}-\eqref{aa2}. Based on the wave-elliptic structure, so we don't need an additional assumption on the vorticity compared with \cite{AK}. Moreover, based on Andersson-Kapitanski's result \cite{AK}, if the Strichartz estimates hold,
	then the flow map can preserve $H^{s+1}$-diffeomorphisms, and the inverse
	map has the same regularity. Therefore, the solution can go back in Euler picture.
\end{remark}

\subsection{Notations}
If f and g are two functions, we say $f \lesssim g$ if and only if there exists a constant
$C>0$ such that $f\leq C g$. We say $f\approx g$ if and only if there exits a constant $C_1, C_2>0$ such that $ C_1 f \leq g \leq C_2 f$. The constant $C$ may change from line to line.

Fourier transforms on $\mathbb{R}^{1+n}$ are denoted by $\hat{\cdot}$ and $\tilde{\cdot}$:
\begin{equation*}
\hat{f}(\xi)=\int_{\mathbb{R}^n} \text{e}^{\text{i}\bx\cdot \xi}f(\bx)d\bx, \quad 	\tilde{F}(\tau,\xi)=\int_\mathbb{R} \int_{\mathbb{R}^n} \text{e}^{\text{i}(t\tau+\bx\cdot \xi)}F(t,\bx)d\bx dt.
\end{equation*}

For $s, \theta \in \mathbb{R}$, let $H^{s,\theta}$ be the space\footnote{The space $\mathcal{S}'(\mathbb{R}^{1+n})$ is the dual space of Schwartz functions.}
\begin{equation*}
	H^{s,\theta}=\left\{ u\in \mathcal{S}'(\mathbb{R}^{1+n}): \left< \xi \right>^s \left<||\tau|-|\xi||\right>^{\theta} \tilde{u}(\tau,\xi) \in L^2(\mathbb{R}^{1+n}) \right\},
\end{equation*}
where we set $\left< \xi \right>=1+|\xi|$ and  $\left<||\tau|-|\xi||\right>=1+ ||\tau|-|\xi||$. For $\theta>\frac12$, then $H^{s,\theta} \hookrightarrow H^s$(c.f. \cite{Selberg}). We use the notation $\|f\|_{s,\theta}$ to denote a norm in $H^{s,\theta}$, that is
\begin{equation}\label{HS1}
	\|u\|_{s,\theta}=\| \left< \xi \right>^s \left<||\tau|-|\xi||\right>^{\theta} \tilde{u}(\tau,\xi) \|_{ L^2(\mathbb{R}^{1+n}) }.
\end{equation}
We also introduce a norm
\begin{equation}\label{HS2}
|f|_{s,\theta}=\|f\|_{s,\theta}+ \|\partial_t f\|_{s-1,\theta}.
\end{equation}

The operator $\Lambda$ and $D$ are two "elliptic" operators on $\mathbb{R}^n$, and $\Lambda_{-}$ is a "elliptic" operator on $\mathbb{R}^{1+n}$, which are denoted by
\begin{equation*}
	\begin{split}
		&\widehat{\Lambda^\alpha f}(\xi)=\left< \xi \right>^\alpha \hat{f}(\xi),
		\\
		& \widehat{D^\alpha f}(\xi)=|\xi|^\alpha \hat{f}(\xi),
		\\
		&\widehat{\Lambda_{\pm}^\alpha F}(\tau,\xi)=\left<||\tau|\pm|\xi||\right>^\alpha \tilde{F}(\tau,\xi),
	\end{split}
\end{equation*}
The space $H^s(\mathbb{R}^n)$ is defined by
\begin{equation*}
	H^{s}(\mathbb{R}^n)=\left\{ u\in \mathcal{S}'(\mathbb{R}^{n}): \left< \xi \right>^s  \hat{u}(\xi) \in L^2(\mathbb{R}^{n}) \right\}.
\end{equation*}
Set a cut-off function $\chi$ satisfying
\begin{equation}\label{chi}
	\chi \in C^\infty_c(\mathbb{R}), \quad \chi=1 \ \text{on} \ [-1,1], \quad \text{supp} \chi \subseteq (-2,2).
\end{equation}
We also set another cut-off function $\phi$ satisfying
\begin{equation}\label{phi}
 	\phi \in C^\infty_c(\mathbb{R}), \quad \phi=1 \ \text{on} \ [-2,2], \quad \text{supp} \phi \subseteq (-4,4).
\end{equation}
\subsection{Organization of the paper}
In Section 2, we introduce a reduction of system \eqref{aa1} and some basic lemmas and propositions. Section 3 presents the proof of Theorem \ref{thm} and Theorem \ref{thm2}.

\section{Preliminaries}
In this part, we will introduce some preliminaries, such as, a reduction of system \eqref{aa1}, some useful lemmas and propositions. These play an important role in our paper.
\subsection{Reformulations of system \eqref{aa1}}

Operating divergence on \eqref{in11}, then we get
\begin{equation}\label{in12}
	-\Delta_x \bar{p}=\text{div}\{ \partial_t \bar{v}^i + \bar{v}^j \partial_j \bar{v}^i- \bar{F}^{kb}\partial_k \bar{F}^{ib} \} .
\end{equation}
Combining \eqref{in7} and \eqref{in10}, \eqref{in12} becomes to
\begin{equation}\label{in13}
	\Delta_x \bar{p}=-\partial_i \bar{v}^m \partial_m \bar{v}^i+ \partial_i \bar{F}^{kb} \partial_k \bar{F}^{ib}.
\end{equation}
We record respectively $\bv$ and $\bF$ being the formulation of $\bar{\bv}$ and $\bar{\bF}$ in Lagrangian coordinates, that is
\begin{equation}\label{in15}
	\bv(t,\by)=\bar{\bv}(t,\bx(t,\by)), \qquad \bF(t,\by)=\bar{\bF}(t,\bx(t,\by)).
\end{equation}
Therefore, in Lagrangian coordinates, \eqref{in13} is transformed to
\begin{equation}\label{A01}
	\begin{split}
		\Delta_x p=& - \frac{\partial v^m}{\partial y_k} (\bF^{-1})^{ki} \cdot \frac{\partial v^i}{\partial y_j} (\bF^{-1})^{jm}
		+ \frac{\partial F^{im}}{\partial y_k} (\bF^{-1})^{kl} \cdot \frac{\partial F^{il}}{\partial y_j} (\bF^{-1})^{jm}
		\\
		=& (\bF^{-1})^{kl} (\bF^{-1})^{mj} \left( - \frac{\partial v^m}{\partial y_k}  \frac{\partial v^l}{\partial y_j}
		+ \frac{\partial F^{im}}{\partial y_k} \frac{\partial F^{il}}{\partial y_j} \right).
	\end{split}
\end{equation}

On the other hand, we also note
\begin{equation}\label{A02}
	\begin{split}
		\frac{\partial v^i}{\partial y_j} =\frac{\partial }{\partial y_j} ( \frac{\partial x^i}{\partial t} )=\frac{\partial }{\partial t} ( \frac{\partial x^i}{\partial y_j} )=\frac{\partial F^{ij}}{\partial t} .
	\end{split}
\end{equation}
Using \eqref{A02}, we can calculate \eqref{A01} by
\begin{equation}\label{a03}
	\begin{split}
		\Delta_x p=& (\bF^{-1})^{kl}(\bF^{-1})^{mj} \left( - \frac{\partial v^m}{\partial y_k}  \frac{\partial v^l}{\partial y_j}
		+ \frac{\partial F^{im}}{\partial y_k} \frac{\partial F^{il}}{\partial y_j} \right)
		\\
		=& (\bF^{-1})^{kl}(\bF^{-1})^{mj} \left( - \frac{\partial F^{mk}}{\partial t}  \frac{\partial F^{jl}}{\partial t}
		+  \frac{\partial F^{mk}}{\partial y_i} \frac{\partial F^{jl}}{\partial y_i} \right).
	\end{split}
\end{equation}
Taking derivatives $\frac{\partial }{\partial y^a}$ on \eqref{in4}, and using \eqref{in8}, \eqref{in15}, we then get
\begin{equation}\label{a04}
	-\frac{\partial^2 F^{ia}}{\partial t^2} + \frac{\partial^2 F^{ia}}{\partial y^b \partial y^b} =\frac{\partial }{\partial y^a} ( \frac{\partial p}{\partial x^i} ).
\end{equation}
Using \eqref{in8} and \eqref{in15} again, we obtain
\begin{equation}\label{a05}
	\text{det}(\frac{\partial \bx}{\partial \by})=1, \quad \Leftrightarrow 	\quad \text{det}\bF=1.
\end{equation}
Using \eqref{a05}, we can calculate
\begin{equation}\label{a06}
	(\bF^{-1})^{ia} = \frac{1}{(n-1)!} \epsilon_{i i_2 i_3 \cdots i_n }\epsilon^{a a_2 a_3 \cdots a_n } F^{i_2 a_2} F^{i_3 a_3} \cdots F^{i_n a_n}.
\end{equation}

After above computations, we
have reformulated \eqref{aa1} to a system of a wave-elliptic equation as follows:
\begin{equation}\label{wavep}
	\begin{cases}
		& \square U^i=\frac{\partial p}{\partial x^i} ,
		\\
		& \square F^{ia}=\frac{\partial }{\partial y^a} ( \frac{\partial p}{\partial x^i} ) ,
		\\
		& \text{det}\bF=1,
		\\
		&  \Delta_x p=(\bF^{-1})^{kl}(\bF^{-1})^{mj} \left( - \frac{\partial F^{mk}}{\partial t}  \frac{\partial F^{jl}}{\partial t}
		+ \frac{\partial F^{mk}}{\partial y_i} \frac{\partial F^{jl}}{\partial y_i} \right).
	\end{cases}
\end{equation}
where we define the operator $\square= -\frac{\partial^2}{\partial t^2}+  \Delta_y$, and $\Delta_y=\frac{\partial^2 }{\partial y^b \partial y^b}$, and $\Delta_x=\textstyle{\sum}_{i=1}^3 \partial^2_i$.

To be clear, we define the martrix $\bG$ by
\begin{equation}\label{a07}
	\bG=\bF- \bE,
\end{equation}
where $\bE$ is the $n\times n$ identity matrix. Inserting \eqref{a07} to \eqref{a06}, we can see
\begin{equation}\label{a15}
		\begin{split}
		(\bF^{-1})^{ia} = & \frac{1}{(n-1)!} \epsilon_{i i_2 i_3 \cdots i_n }\epsilon^{a a_2 a_3 \cdots a_n } F^{i_2 a_2}  \cdots F^{i_n a_n}
			\\
			=& \frac{1}{(n-1)!} \epsilon_{i i_2 i_3 \cdots i_n }\epsilon^{a a_2 a_3 \cdots a_n } (\delta^{i_2 a_2}+ G^{i_2 a_2}) \cdots (\delta^{i_n a_n}+ G^{i_n a_n}).
		\end{split}
\end{equation}
By \eqref{a15}, we can conclude
\begin{equation}\label{a10}
	\begin{split}
		| \bF^{-1} | \leq C(1+ | \bG | + | \bG |^2+ \cdots + | \bG |^{n-1} ).
	\end{split}
\end{equation}

Inserting \eqref{a07} to \eqref{wavep}, then $\bG$ satisfies 
\begin{equation}\label{a08}
	\begin{cases}
		& \square G^{ia}=\frac{\partial }{\partial y^a} ( \frac{\partial p}{\partial x^i} ) ,
		\\
		& \Delta_x p=(\bF^{-1})^{kl}(\bF^{-1})^{mj} \cdot Q_0(G^{mk}, G^{jl}) ,
		\\
		& \text{det}\bF=1.
	\end{cases}
\end{equation}

The initial condition \eqref{aa2} also tells us
\begin{equation}\label{a09}
	\begin{split}
		& G^{ia}|_{t=0}=\frac{\partial U^{ia}}{\partial y^a}|_{t=0}=\frac{\partial U_0^{ia}}{\partial y^a},
		\\
		& \frac{\partial G^{ia}}{\partial t}|_{t=0}=\frac{\partial}{\partial t}\left( \frac{\partial U^{i}}{\partial y^a} \right)|_{t=0}=\frac{\partial}{\partial y^a}\left( \frac{\partial U^{i}}{\partial t} \right)|_{t=0}=\frac{\partial v_0^{i}}{\partial y^a}.
	\end{split}
\end{equation}
Next, let us introduce some useful product estimate and commutator estimates, which plays an important role in the paper.
\subsection{Useful lemmas}
\begin{Lemma}[\cite{Selberg},Theorem 12, Theorem 13]\label{nonlinearE}
Assume $s\in \mathbb{R}$, $\theta \in (\frac12,1)$, $\varepsilon \in [0,1-\theta]$. Consider the Cauchy problem for the linear wave equation
\begin{equation}\label{linearw}
	\begin{cases}
	& \square u=F, \quad (t,x)\in \mathbb{R}^{1+n},
	\\
	& u|_{t=0} =f, \quad  \partial_t u|_{t=0} = g.
	\end{cases}
\end{equation}
Let $f,g$ and $F$ satisfy $f\in H^s$, $g\in H^{s-1}$,  and $F\in H^{s-1,\theta+\varepsilon-1}$. Let $\chi$ and $\phi$ be stated in \eqref{chi} and \eqref{phi}. Let $0<T<1$ and define
\begin{equation}\label{defu}
  u(t)=\chi(t)u_0+ \chi(\frac{t}{T})u_1+u_2,
\end{equation}
where
\begin{equation}\label{defu0}
 \begin{split}
 	u_0= &\cos(tD)f +  D^{-1} \sin(tD)g,
 \\
 	F_1= & \phi( T^{\frac12}\Lambda_{-} )F, \quad F_2= ( 1-\phi( T^{\frac12}\Lambda_{-} ) )F,
 	\\
 	u_1= & \int^t_0  D^{-1} \sin( (t-t')D ) F_1(t')dt',
 	\\
 	u_2= & \square^{-1} F_2,
 \end{split}
\end{equation}
Then, the function $u$ defined in \eqref{defu}-\eqref{defu0} satisfies the following estimate
\begin{equation}\label{none1}
|u|_{s,\theta} \leq C(\|f\|_{H^s}+ \|g\|_{H^{s-1}}+ T^{\frac{\varepsilon}{2}}\| F\|_{s-1,\theta+\varepsilon-1} ),
\end{equation}
where $C$ only depends on $\chi$ and $\theta$. Moreover, $u$ is the unique solution of \eqref{linearw} on $[0,T]\times \mathbb{R}^n$ such that $u\in C([0,T];H^s) \cap C^1([0,T];H^{s-1})$.
\end{Lemma}
\begin{remark}
	For $\varepsilon=0$ or $\varepsilon\in (0,1-\theta]$, we refer the readers to Selberg's paper \cite{Selberg} Theorem 12 and Theorem 13 respectively. 
\end{remark}
Let $h$ and $w$ be a solution of
\begin{equation}\label{k1}
	\begin{cases}
			& \square h=0, \quad (t,x)\in \mathbb{R}^{1+n},
		\\
		& u|_{t=0} =f_1, \quad  \partial_t u|_{t=0} = g_1,
	\end{cases}
\end{equation}
and
\begin{equation}\label{k2}
	\begin{cases}
		& \square w=0, \quad (t,x)\in \mathbb{R}^{1+n},
		\\
		& u|_{t=0} =f_2, \quad  \partial_t u|_{t=0} = g_2.
	\end{cases}
\end{equation}
\begin{Lemma}[\cite{FK,KS2,Selberg}]\label{bilinearE}
	Assume $n\geq 2$ and $s>\frac{n}{2}$. Let $\theta \in (\frac12,1)$, $\varepsilon \in (0,1-\theta]$ and $\frac{n-1}{2}+\theta+\varepsilon < s$. Let $h$ and $w$ satisfy \eqref{k1}-\eqref{k2}. Then
	\begin{equation*}
	\|Q_0(h,w)\|_{s-1,\theta+\varepsilon-1} \lesssim |h|_{s,\theta}|w|_{s,\theta}.
	\end{equation*}
\end{Lemma}
\begin{remark}
	For $\varepsilon \in (0,1-\theta)$, please refer to Section 4.1.3 in \cite{Selberg} (see also Section 7.2 in \cite{KS2}). For $\varepsilon =1-\theta$, we refer the readers to \cite{FK} (cf. Corollary 13.3 in \cite{FK}). Then we obtain $\|Q_0(h,w)\|_{s-1,0} \lesssim \|h\|_{s,0}\|w\|_{s,0}$. Using the fact $H^{s,\theta}\subseteq H^s$ for $\theta \in (\frac12,1)$, then we get $\|Q_0(h,w)\|_{s-1,0} \lesssim |h|_{s,\theta}|w|_{s,\theta}$.
\end{remark}
\begin{Lemma}[\cite{Selberg},Proposition 10]\label{prodE}
	Let $s_j\geq 0$ and $\theta_j \geq 0$ for $1\leq j \leq 3$. If
	\begin{equation*}
	s_1+s_2+s_3 > \frac{n}{2}, \quad \theta_1+\theta_2+\theta_3 > \frac{1}{2},
	\end{equation*}
then we have
\begin{equation*}
	\| f_1 f_2 \|_{-s_1, -\theta_1} \leq C \| f_1 \|_{s_2, \theta_2} \| f_2 \|_{s_3, \theta_3}.
\end{equation*}
\end{Lemma}
Next, we will introduce some propositions, which plays an important role in the paper.
\subsection{Some propositions}
\begin{proposition}\label{LD2}
Assume $s>\frac{n}{2}$, $\theta \in(\frac12,1)$ and $\varepsilon \in(0,1-\theta]$. Then
	\begin{equation}\label{LDe0}
		\| f_1 f_2 \|_{s-1,\theta+\varepsilon-1 } \leq C \|f_1 \|_{s-1,\theta} \|f_2 \|_{s,\theta+\varepsilon-1}.
	\end{equation}
\end{proposition}
\begin{proof}
If $\varepsilon=1-\theta$, using $s>\frac{n}{2}, \theta \in(\frac12,1)$, then we have
\begin{equation}\label{T3}
	\| f_1 f_2 \|_{s-1,0} \lesssim \|f_1 \|_{s-1,0} \|f_2 \|_{s,0} \lesssim \|f_1 \|_{s-1,\theta} \|f_2 \|_{s,0}.
\end{equation}
If $\varepsilon< 1-\theta$, using the triangle inequality on the frequency side, it's not hard to see the following estimate
\begin{equation*}
	\Lambda^{a}( f_1 f_2 )  \lesssim \Lambda^{a} f_1 \cdot f_2 + \Lambda^{a} f_2 \cdot f_1, \quad a>0.
\end{equation*}
So we have	
\begin{equation}\label{LDe1}
\begin{split}
\| f_1 f_2 \|_{s-1,\theta+\varepsilon-1} = & C \|\Lambda^{s-1}( f_1 f_2 ) \|_{0,\theta+\varepsilon-1}
\\
\leq & C  ( \|\Lambda^{s-1} f_1 \cdot f_2  \|_{0,\theta+\varepsilon-1}+ \|\Lambda^{s-1} f_2 \cdot f_1  \|_{0,\theta+\varepsilon-1} ).
\end{split}
\end{equation}
By Lemma \ref{prodE} (taking $s_1=0$, $\theta_1=1-\theta-\varepsilon$, $s_2=0$, $\theta_2=\theta$, and $s_3=s$, $ \theta_3=\theta+\varepsilon-1$), we get
\begin{equation}\label{LDe2}
\begin{split}
 \|\Lambda^{s-1} f_1 \cdot f_2  \|_{0,\theta+\varepsilon-1} \leq & C\|\Lambda^{s-1} f_1  \|_{0,\theta} \| f_2  \|_{s,\theta+\varepsilon-1}
 \\
  \leq & C\| f_1  \|_{s-1,\theta} \| f_2  \|_{s,\theta+\varepsilon-1} .
\end{split}
\end{equation}
Using Lemma \ref{prodE} again (taking $s_1=0$, $\theta_1=1-\theta-\varepsilon$, $s_2=s-1$, $\theta_2=\theta$, and $s_3=1$, $ \theta_3=\theta+\varepsilon-1$), we can derive
\begin{equation}\label{LDe3}
\begin{split}
\|\Lambda^{s-1} f_2 \cdot f_1  \|_{0,\theta+\varepsilon-1} \leq &  C\| f_1  \|_{s-1,\theta} \|\Lambda^{s-1} f_2  \|_{1,\theta+\varepsilon-1}
\\
  \leq & C\| f_1  \|_{s-1,\theta} \| f_2  \|_{s,\theta+\varepsilon-1} .
\end{split}
\end{equation}
Combining \eqref{LDe1}-\eqref{LDe3}, then \eqref{LDe0} holds when $\varepsilon< 1-\theta$. In a result, we complete the proof of Proposition \ref{LD2}.
\end{proof}
\begin{proposition}\label{LD3}
	Assume $s>\frac{n}{2}$, $\theta \in(\frac12,1)$ and $\varepsilon \in(0,1-\theta]$. Then
	\begin{equation}\label{LDe4}
		\| f_1 f_2 \|_{s,\theta+\varepsilon-1 } \leq C \|f_1 \|_{s,\theta} \|f_2 \|_{s,\theta+\varepsilon-1}.
	\end{equation}
\end{proposition}
\begin{proof}
If $\varepsilon=1-\theta$, using $s>\frac{n}{2}, \theta \in(\frac12,1)$, then we have
\begin{equation}\label{T2}
	\| f_1 f_2 \|_{s,0} \lesssim \|f_1 \|_{s,0} \|f_2 \|_{s,0} \lesssim \|f_1 \|_{s,\theta} \|f_2 \|_{s,0}.
\end{equation}
If $\varepsilon<1-\theta$,
using the triangle inequality on the frequency side, it's not hard to see the following estimate
\begin{equation*}
	\Lambda^{a}( f_1 f_2 )  \lesssim \Lambda^{a} f_1 \cdot f_2 + \Lambda^{a} f_2 \cdot f_1, \quad a>0.
\end{equation*}	
So we have
	\begin{equation}\label{LDe5}
		\begin{split}
			\| f_1 f_2 \|_{s,\theta+\varepsilon-1} = & C \|\Lambda^{s}( f_1 f_2 ) \|_{0,\theta+\varepsilon-1}
			\\
			\leq & C  ( \|\Lambda^{s} f_1 \cdot f_2  \|_{0,\theta+\varepsilon-1}+ \|\Lambda^{s} f_2 \cdot f_1  \|_{0,\theta+\varepsilon-1} ).
		\end{split}
	\end{equation}
	By Lemma \ref{prodE} (taking $s_1=0$, $\theta_1=1-\theta-\varepsilon$, $s_2=0$, $\theta_2=\theta$, and $s_3=s$, $ \theta_3=\theta+\varepsilon-1$), we get
	\begin{equation}\label{LDe6}
		\begin{split}
			\|\Lambda^{s} f_1 \cdot f_2  \|_{0,\theta+\varepsilon-1} \leq & C\|\Lambda^{s} f_1  \|_{0,\theta} \| f_2  \|_{s,\theta+\varepsilon-1}
			\\
			\leq & C\| f_1  \|_{s,\theta} \| f_2  \|_{s,\theta+\varepsilon-1} .
		\end{split}
	\end{equation}
	Similarly, using Lemma \ref{prodE} (taking $s_1=0$, $\theta_1=1-\theta-\varepsilon$, $s_2=s-1$, $\theta_2=\theta$, and $s_3=1$, $ \theta_3=\theta+\varepsilon-1$) again, we can prove that
	\begin{equation}\label{LDe7}
		\begin{split}
			\|\Lambda^{s} f_2 \cdot f_1  \|_{0,\theta+\varepsilon-1} \leq &  C\| f_1  \|_{s,\theta} \|\Lambda^{s} f_2  \|_{1,\theta+\varepsilon-1}
			\\
			\leq & C\| f_1  \|_{s,\theta} \| f_2  \|_{s,\theta+\varepsilon-1} .
		\end{split}
	\end{equation}
	Gathering \eqref{LDe5}-\eqref{LDe7}, we can get \eqref{LDe4} when $\varepsilon<1-\theta$. So we have finished the proof of Proposition \ref{LD3}.
\end{proof}
\begin{proposition}\label{LD4}
	Assume $0<s<1$. Let $(t,\bx)$ be Euler coordinate, and $(t,\by)$ be Lagrange coordinate. Let $u$ be a function from $(t,\by) \rightarrow \mathbb{R}^n$. Let $\bar{u}(t,\bx)=u(t,\by(t,\bx))$. If $\text{det}(\frac{\partial \bx}{\partial \by})=1$, then we have
	\begin{equation}\label{ab00}
			\| u \|_{H^s(\mathbb{R}^n_y)} \leq C \| \frac{\partial \bx}{\partial \by} \|^{\frac{n}{2}+s}_{L^\infty(\mathbb{R}^n_x)} \|\bar{u} \|_{H^s(\mathbb{R}^n_x)},
	\end{equation}
	and
	\begin{equation}\label{ab01}
 \| \bar{u}\|_{H^s(\mathbb{R}^n_x)} \leq C \| \frac{\partial \by}{\partial \bx} \|^{\frac{n}{2}+s}_{L^\infty(\mathbb{R}^n_y)} \| u \|_{H^s(\mathbb{R}^n_y)}.
	\end{equation}
\end{proposition}
\begin{proof}
Firstly, by changing of coordinates, we have
\begin{equation}\label{ab0}
	\begin{split}
		\| u(t,\cdot) \|_{L^2(\mathbb{R}^n)} &= \int_{\mathbb{R}^n_y} | u(t,\by) |^2 d\by
		\\
		& = \int_{\mathbb{R}^n} | \bar{u}(t,\bx) |^2  \text{det}(\frac{\partial \by}{\partial \bx}) d\bx
		\\
		& = \| \bar{u} \|_{L^2(\mathbb{R}^n_x)}.
	\end{split}
\end{equation}
For the homogeneous norm $\dot{H}^s$, cf. \cite{BCD}, we have
\begin{equation*}
	\begin{split}
		\| u(t,\cdot)  \|^2_{\dot{H}^s (\mathbb{R}^n_y)} &= \int_{\mathbb{R}^n} \int_{\mathbb{R}^n} \frac{|u(t,\by+\bh)-u(t,\by)|^2}{|\bh|^{n+2s}} d\by d\bh.
	\end{split}
\end{equation*}
So we also obtain
\begin{equation}\label{ab1}
	\begin{split}
		\| u \|^2_{\dot{H}^s (\mathbb{R}^n_y)} &= \int_{\mathbb{R}^n} \int_{\mathbb{R}^n} \frac{|\bar{u}(t,x(t,\by+\bh))-\bar{u}(t,x(t,\by))|^2}{|x(t,\by+\bh)-x(t,\by)|^{n+2s}} \cdot \frac{|x(t,\by+\bh)-x(t,\by)|^{n+2s} }{|\bh|^{n+2s} } d\by d\bh
		\\
		&\leq \| \frac{\partial \bx}{\partial \by} \|^{n+2s}_{L^\infty(\mathbb{R}^n_x)} \cdot \int_{\mathbb{R}^n} \int_{\mathbb{R}^n} \frac{|\bar{u}(t,x(t,\by+\bh))-\bar{u}(t,x(t,\by))|^2}{|x(t,\by+\bh)-x(t,\by)|^{n+2s}} d\by d\bh .
	\end{split}
\end{equation}
On the other hand, by changing of coordinates, it yields
\begin{equation}\label{ab2}
	\begin{split}
		& \int_{\mathbb{R}^n} \int_{\mathbb{R}^n} \frac{|\bar{u}(t,x(t,\by+\bh))-\bar{u}(t,x(t,\by))|^2}{|x(t,\by+\bh)-x(t,\by)|^{n+2s}} d\by d\bh 
		 \\
		 =& \int_{\mathbb{R}^n} \int_{\mathbb{R}^n} \frac{|\bar{u}(t,x(t,\bar{\by}))-\bar{u}(t,x(t,\by))|^2}{|x(t,\bar{\by})-x(t,\by)|^{n+2s}} d\bar{\by} d\by 
		 \\
		 =& \int_{\mathbb{R}^n} \int_{\mathbb{R}^n} \left( \frac{|\bar{u}(t,x(t,\bar{\by}))-\bar{u}(t,\bx)|^2}{|x(t,\bar{\by})-\bx|^{n+2s}} \right)  \text{det}(\frac{\partial \bx}{\partial \by}) d\bar{\by} d\bx 
		 \\
		= & \int_{\mathbb{R}^n} \int_{\mathbb{R}^n} \left( \frac{|\bar{u}(t,\bar{\bx})-\bar{u}(t,\bx)|^2}{|\bar{\bx}-\bx|^{n+2s}} \right)  \text{det}(\frac{\partial \bx(\bar{\by})}{\partial \bar{\by}}) d{\bar{\bx}} d\bx 
		 \\
		 	= & \int_{\mathbb{R}^n} \int_{\mathbb{R}^n}  \frac{|\bar{u}(\bar{\bx})-\bar{u}(\bx)|^2}{|\bar{\bx}-\bx|^{n+2s}}  d{\bar{\bx}} d\bx .
	\end{split}
\end{equation}
where we set $x(t,\bar{\by})=\bar{\bx}$. Combining \eqref{ab0}, \eqref{ab1}, and \eqref{ab2}, we can conclude the conclusion \eqref{ab00}. Similarly, we can also get \eqref{ab01}.
\end{proof}
\section{Proof of Theorem \ref{thm}, Theorem \ref{thmb} and Theorem \ref{thm2}}
\subsection{Proof of Theorem \ref{thm}}
We will prove Theorem \ref{thm} by several steps. Since there is only unknown functions $\bG$ in \eqref{a08}, so we firstly consider \eqref{a08} with the initial condition \eqref{a09}. After that, we then study $\bU$ based on the local well-posedness of \eqref{a08}.

Recall\footnote{Here \eqref{a15} is equivalent to $\text{det}\bF=1$.}
\begin{equation}\label{p0}
	\begin{cases}
		& \square G^{ia}=\frac{\partial }{\partial y^a} ( \frac{\partial p}{\partial x^i} ) ,
		\\
		& \Delta_x p=(\bF^{-1})^{kl}(\bF^{-1})^{mj} Q_0(G^{mk}, G^{jl}) ,
		\\
		& 	 (G^{ia}, \frac{\partial G^{ia}}{\partial t})|_{t=0}=(\frac{\partial U_0^{ia}}{\partial y^a}, \frac{\partial v_0^{i}}{\partial y^a}).
	\end{cases}
\end{equation}
where $\bF^{-1}$ satisfies \eqref{a15}.

We will use the contraction mapping principle to establish the existence of solution. We define the working space
\begin{equation}\label{Xd}
	X_{s,\theta}=\left\{ u\in \mathcal{S}'(\mathbb{R}^{1+n}): |u|_{s,\theta}\leq C_1(\|\bU_0\|_{H^{s+1}}+\|\bv_0\|_{H^{s}}) \right\}.
\end{equation}
Here $C_1$ will be defined later. For $\bG \in X_{s,\theta}$, we denote the map $\text{M}$ by
\begin{equation}\label{t0}
\begin{split}
	\text{M} G^{ia}= 	&\chi(t) \left\{ \cos(tD) (\frac{\partial U^i_0}{\partial y^a}) +  D^{-1} \sin(tD) (\frac{\partial v^i_0}{\partial y^a}) \right\} 
	\\
	& + \chi(\frac{t}{T}) \int^t_0 D^{-1} \sin( (t-t')D ) \left\{  \phi( T^{\frac12}\Lambda_{-} ) \frac{\partial }{\partial y^a}(\frac{\partial p}{\partial x_i}) \right\} (t')dt'
	\\
	 &+  \square^{-1} ( 1-\phi( T^{\frac12}\Lambda_{-} ) )  \frac{\partial }{\partial y^a}(\frac{\partial p}{\partial x_i}),
\end{split}	
\end{equation}
where $p$ satisfies (please refer \eqref{a08})
\begin{equation}\label{t00}
	\Delta_x p=(\bF^{-1})^{kl}(\bF^{-1})^{mj}  Q_0(G^{mk}, G^{jl}).
\end{equation}
By using Lemma \ref{nonlinearE}, we shall calculate out
\begin{equation}\label{t1}
	\begin{cases}
		& \square \text{M} G^{ia}=\frac{\partial }{\partial y^a} ( \frac{\partial p}{\partial x^i} ), \quad \text{on} \quad [0,T]\times \mathbb{R}^n,
	\\
	& \text{M} G^{ia}|_{t=0}=\frac{\partial U^i_0}{\partial y^a}, \quad  \frac{\partial } {\partial t} \text{M} G^{ia}|_{t=0}=\frac{\partial v^i_0}{\partial y^a}. 
	\end{cases}
\end{equation}

Using Lemma \ref{nonlinearE} again, we have
\begin{equation}\label{t2}
	\begin{split}
	| \text{M} \bG |_{s,\theta} \leq & C( \|\frac{\partial \bU_0}{\partial y}\|_{H^s}+ \| \frac{\partial \bv_0}{\partial y} \|_{H^{s-1}} + T^{\frac{\varepsilon}{2}}  \| \frac{\partial }{\partial y}(\frac{\partial p}{\partial x}) \|_{s-1,\theta+ \varepsilon -1 }).
	\end{split}
\end{equation}
Note $	\frac{\partial }{\partial y}(\frac{\partial p}{\partial x})= \frac{\partial }{\partial x}(\frac{\partial p}{\partial x}) \cdot \frac{\partial x}{\partial y} $, so we get
\begin{equation}\label{t20}
		\frac{\partial }{\partial y}(\frac{\partial p}{\partial x})
		= \bF\cdot \frac{\partial }{\partial x}(\frac{\partial p}{\partial x})
		= (\bE+\bG)\cdot \frac{\partial }{\partial x}(\frac{\partial p}{\partial x}).
\end{equation}
We then get
\begin{equation}\label{t3}
	\begin{split}
		| \text{M} \bG |_{s,\theta} \leq  & C \left\{ \|\bU_0\|_{H^{s+1}}+\|\bv_0\|_{H^{s}}+ T^{\frac{\varepsilon}{2}}  ( \| \frac{\partial^2 p}{\partial^2 x} \|_{s-1,\theta+ \varepsilon -1 }+ \| \bG \frac{\partial^2 p}{\partial^2 x} \|_{s-1,\theta+ \varepsilon -1 } ) \right\}
		\\
		\leq  & C \left\{  \|\bU_0\|_{H^{s+1}}+\|\bv_0\|_{H^{s}} + T^{\frac{\varepsilon}{2}}  (1+ \| \bG \|_{s,\theta}) \| \frac{\partial^2 p}{\partial^2 x} \|_{s-1,\theta+ \varepsilon -1 } \right\}.
	\end{split}
\end{equation}
For $\theta+ \varepsilon -1 \leq 0$, notice
\begin{equation}\label{T0}
	\| \frac{\partial^2 p}{\partial^2 \bx} \|_{s-1,\theta+ \varepsilon -1 } \leq \| \frac{\partial^2 p}{\partial^2 \bx} \|_{s-1,0}=\| \frac{\partial^2 p}{\partial^2 \bx} \|_{L^2_{t}H^{s-1}(\mathbb{R}^n_y)}.
\end{equation}
Since $s>\frac{n+1}{2}$, using \eqref{T0}, \eqref{t00}, Lemma \ref{LD4}, and Proposition \ref{LD3}, we can see
\begin{equation}\label{t4}
	\begin{split}
	\| \frac{\partial^2 p}{\partial^2 \bx} \|_{s-1,\theta+ \varepsilon -1 } 
	\leq &  C \| \frac{\partial \bx}{\partial \by} \|^{\frac{n}{2}+s}_{L^\infty(\mathbb{R}^n_x)} \| \frac{\partial^2 \bar{p}}{\partial^2 \bx} \|_{L^2_tH^{s-1}(\mathbb{R}^n_x)}
	\\
	\leq & C \| \frac{\partial \bx}{\partial \by} \|^{\frac{n}{2}+s}_{L^\infty(\mathbb{R}^n_x)} 
	\| \Delta_x \bar{p}\|_{L^2_t H^{s-1}(\mathbb{R}^n_x)}
	\\
	\leq & C \| \frac{\partial \bx}{\partial \by} \|^{\frac{n}{2}+s}_{L^\infty(\mathbb{R}^n_y)} 
	\| \Delta_x \bar{p}\|_{L^2_t H^{s-1}(\mathbb{R}^n_y)}
	\\
	\leq & C (1+\|\bG\|^{\frac{n}{2}+s}_{s,\theta} ) (1+ \|\bG\|^{2(n-1)}_{s,\theta} ) \| Q_0(\bG,\bG) \|_{s-1,0}.
	\end{split}
\end{equation}
By Lemma \ref{nonlinearE}, we get
\begin{equation}\label{t5}
	\| Q_0(\bG,\bG) \|_{s-1,0 } \leq C |\bG|^2_{s,\theta}. 
\end{equation}
Combining \eqref{t3}, \eqref{t4} and \eqref{t5}, it yields
\begin{equation}\label{t6}
	\begin{split}
		| \text{M} \bG |_{s,\theta} \leq   & C ( \|\bU_0\|_{H^{s+1}}+\|\bv_0\|_{H^{s}} )
		\\
		& + CT^{\frac{\varepsilon}{2}} (1+\|\bG\|^{\frac{n}{2}+s}_{s,\theta} ) \left(|\bG|^2_{s,\theta}+ |\bG|^3_{s,\theta}+|\bG|^4_{s,\theta}+ \cdots + |\bG|^{2n}_{s,\theta} \right).
	\end{split}
\end{equation}
Taking $C_1=2C$ in \eqref{Xd}, and choosing 
\begin{equation}\label{CT}
	T=\left\{ \frac{M}{(1+(C+1)^{\frac{n}{2}+s}M^{\frac{n}{2}+s})[ (C+1)^2M^2+(C+1)^3M^3+\cdots+(C+1)^{2n}M^{2n} ]} \right\}^{\frac{2}{\varepsilon}},
\end{equation}
then 
\begin{equation}\label{t7}
	\begin{split}
		| \text{M} \bG |_{s,\theta} \leq   & 2C ( \|\bU_0\|_{H^{s+1}}+\|\bv_0\|_{H^{s}} )=C_1( \|\bU_0\|_{H^{s+1}}+\|\bv_0\|_{H^{s}} ).
	\end{split}
\end{equation}
Hence, \eqref{t7} tells us that $\text{M}$ is a map from $X_{s,\theta}$ to $X_{s,\theta}$. Next, we will prove that $\text{M}$ is a contraction map in $X_{s,\theta}$.

Considering
\begin{equation*}
	Q_0(\bG,\bG)-Q_0(\bH,\bH)=Q_0(\bG-\bH,\bG)+Q_0(\bH,\bG-\bH),
\end{equation*}
so we can conclude
\begin{equation}\label{t8}
	\begin{split}
		| \text{M} \bG -  \text{M} \bH |_{s,\theta} \leq   & C  T^{\frac{\varepsilon}{2}}  |\bG - \bH|_{s,\theta} \left(|\bG|_{s,\theta}+ |\bG|^2_{s,\theta}+|\bG|^3_{s,\theta}+ \cdots + |\bG|^{n+1}_{s,\theta} \right)
		\\
		& + C  T^{\frac{\varepsilon}{2}}  |\bG - \bH|_{s,\theta} \left(|\bH|_{s,\theta}+ |\bH|^2_{s,\theta}+|\bH|^3_{s,\theta}+ \cdots + |\bH|^{n+1}_{s,\theta} \right).
	\end{split}
\end{equation}
When $T$ is sufficiently small, using \eqref{t8}, we shall get 
\begin{equation}\label{t9}
	\begin{split}
		| \text{M} \bG -  \text{M} \bH |_{s,\theta} \leq   & \frac12  |\bG - \bH |_{s,\theta}.
	\end{split}
\end{equation}
Therefore, $\text{M}$ is a contraction mapping in the space $X_{s,\theta}$. Using the contraction mapping principle, we prove that there is a unique solution satisfying \eqref{a08}-\eqref{a09}. Next, we will prove the continuous dependence on initial data for \eqref{p0}. Set $\bG_1$ and $\bG_2$ satisfying
\begin{equation}\label{G1}
	\begin{cases}
		& \square G_1^{ia}=\frac{\partial }{\partial y^a} ( \frac{\partial p_1}{\partial x^i} ) ,
		\\
		& \Delta_x p_1=(\bF_1^{-1})^{kl}(\bF_1^{-1})^{mj} Q_0(G_1^{mk}, G_1^{jl}) ,
		\\
		& 	 (G_1^{ia}, \frac{\partial G_1^{ia}}{\partial t})|_{t=0}=(\frac{\partial U_{01}^{ia}}{\partial y^a}, \frac{\partial v_{01}^{i}}{\partial y^a}),
	\end{cases}
\end{equation}
and
\begin{equation}\label{G2}
	\begin{cases}
		& \square G_2^{ia}=\frac{\partial }{\partial y^a} ( \frac{\partial p_2}{\partial x^i} ) ,
		\\
		& \Delta_x p_2=(\bF_2^{-1})^{kl}(\bF_2^{-1})^{mj} Q_0(G_2^{mk}, G_2^{jl}) ,
		\\
		& 	 (G_2^{ia}, \frac{\partial G_2^{ia}}{\partial t})|_{t=0}=(\frac{\partial U_{02}^{ia}}{\partial y^a}, \frac{\partial v_{02}^{i}}{\partial y^a}),
	\end{cases}
\end{equation}
where $\bF_1^{-1}$ and $\bF_2^{-1}$ satisfy
\begin{equation}\label{G3}
	\begin{split}
		(\bF_1^{-1})^{ia} = & \frac{1}{(n-1)!} \epsilon_{i i_2 i_3 \cdots i_n }\epsilon^{a a_2 a_3 \cdots a_n } (\delta^{i_2 a_2}+ G_1^{i_2 a_2}) \cdots (\delta^{i_n a_n}+ G_1^{i_n a_n}),
		\\
		(\bF_2^{-1})^{ia} = & \frac{1}{(n-1)!} \epsilon_{i i_2 i_3 \cdots i_n }\epsilon^{a a_2 a_3 \cdots a_n } (\delta^{i_2 a_2}+ G_2^{i_2 a_2}) \cdots (\delta^{i_n a_n}+ G_2^{i_n a_n}).
	\end{split}
\end{equation}
Using Lemma \ref{nonlinearE}, we shall prove
\begin{equation}\label{G4}
	\begin{split}
		|\bG_2-\bG_1|_{s,\theta} \leq C & \left( \|\bU_{02}-\bU_{01}\|_{H^{s+1}} +\|\bv_{02}-\bv_{01}\|_{H^{s+1}} \right)
		\\
		& +CT^{\frac{\varepsilon}{2} } |\bG_2 - \bG_1|_{s,\theta} \left(|\bG_1|_{s,\theta}+ |\bG_1|^2_{s,\theta}+|\bG_1|^3_{s,\theta}+ \cdots + |\bG_1|^{n+1}_{s,\theta} \right)
		\\
		& +CT^{\frac{\varepsilon}{2} } |\bG_2 - \bG_1|_{s,\theta} \left(|\bG_2|_{s,\theta}+ |\bG_2|^2_{s,\theta}+|\bG_2|^3_{s,\theta}+ \cdots + |\bG_2|^{n+1}_{s,\theta} \right).
	\end{split}
\end{equation}
By \eqref{G4}, for $\bG_1, \bG_2\in X_{s,\theta}$, if $T$ is sufficiently small, then we can get
 \begin{equation}\label{G5}
 	\begin{split}
 		|\bG_2-\bG_1|_{s,\theta} \leq C & \left( \|\bU_{02}-\bU_{01}\|_{H^{s+1}} +\|\bv_{02}-\bv_{01}\|_{H^{s+1}} \right).
 	\end{split}
 \end{equation}
Therefore, \eqref{G5} tells us
 \begin{equation}\label{G6}
	\begin{split}
		& \|\bG_2-\bG_1\|_{L^\infty_{[0,T]}H^s}+\|\frac{\partial (\bG_2-\bG_1 )}{\partial t}\|_{L^\infty_{[0,T]}H^{s-1}} 
		\\
		\leq & C \left( \|\bU_{02}-\bU_{01}\|_{H^{s+1}} +\|\bv_{02}-\bv_{01}\|_{H^{s+1}} \right).
	\end{split}
\end{equation}
As a result, we have proved the continuous dependence of solutions on the initial data.

We still need to prove the existence of solution $\bU$ in \eqref{wavep}. Using Lemma \ref{nonlinearE} again, we can see that there is a unique solution $\bU$ such that
\begin{equation}\label{tt0}
	\begin{cases}
		& \square U^i=\frac{\partial p}{\partial x^i}, \quad (t,\by)\in [0,T]\times \mathbb{R}^n,
		\\
		& ( U^i, \frac{\partial U^i}{\partial t} )|_{t=0}=(U^i_0,v^i_0).
	\end{cases}
\end{equation}
Moreover, for $\theta \in (\frac12,1)$, the function $\bU$ satisfies the following bound
\begin{equation*}
	| \bU |_{s+1,\theta} \leq C(\|\bU_0\|_{H^{s+1}}+\|\bv_0\|_{H^{s}}+ T^{\frac{\varepsilon}{2}} | \frac{\partial p}{\partial x^i} |_{s,\theta+\varepsilon-1}).
\end{equation*}
Using \eqref{t20}, \eqref{t4}, and \eqref{CT}, we can see
\begin{equation}\label{tt1}
	| \bU |_{s+1,\theta} \leq 2C(\|\bU_0\|_{H^{s+1}}+\|\bv_0\|_{H^{s}}).
\end{equation}
Therefore, $\bU \in C( {[0,T],H^{s+1}} ) \cap C^1( {[0,T],H^{s}}) $. Similarly, for the system \eqref{tt0}, the continuous dependence on the initial data also holds.
\subsection{Proof of Theorem \ref{thm2}}
In the case of $s>s_0(n)$, we will use continuous induction argument. Set
\begin{equation}\label{Tstar}
	T_*=\min\{  (\frac{1}{54C^2(\|\bU_0\|_{H^{s+1}}+\|\bv_0\|_{H^{s}})^3} )^{\frac43},  (\frac{1}{486C^2(\|\bU_0\|_{H^{s+1}}+\|\bv_0\|_{H^{s}})^4} )\}.
\end{equation}
We assume that
\begin{equation}\label{X2}
\|\bG\|_{L^\infty_{[0,T_*]} H^{s}}+\|\frac{\partial \bG}{\partial t}\|_{L^\infty_{[0,T_*]} H^{s-1}} \leq 3( \|\bU_0\|_{H^{s+1}}+\|\bv_0\|_{H^{s}} ) ,
\end{equation}
and
\begin{equation}\label{X3}
	\|\bU\|_{L^\infty_{[0,T_*]} H^{s+1}}+\|\frac{\partial \bU}{\partial t}\|_{L^\infty_{[0,T_*]} H^{s}} \leq 3( \|\bU_0\|_{H^{s+1}}+\|\bv_0\|_{H^{s}} ) .
\end{equation}
Since the Strichartz estimates of linear wave equations is different when $n=2,3$. So we divide it into two cases.

\textit{Case 1:} $n=2$. For $s>\frac74$, using Strichartz estimates of linear wave equations (cf. \cite{KeelTao}), it yields
\begin{equation}\label{X4}
	\begin{split}
	 \| d\bG  \|_{ L^{4}_{[0,T_*]}L^\infty} 
	\leq & C  ( \|\bU_0\|_{H^{s+1}}+\|\bv_0\|_{H^{s}} + \|\frac{\partial }{\partial y}(\frac{\partial p}{\partial x})\|_{ L^{1}_{[0,T_*]} H^{s-1} }  )
	\\
		\leq & C  ( \|\bU_0\|_{H^{s+1}}+\|\bv_0\|_{H^{s}} + (T_*)^{\frac34} \|\frac{\partial }{\partial y}(\frac{\partial p}{\partial x})\|_{ L^{4}_{[0,T_*]} H^{s-1} }  ).
	\end{split}
\end{equation}
Using \eqref{t20}, \eqref{t00}, \eqref{X2}, \eqref{Tstar} and H\"older's inequality, so we obtain
\begin{equation}\label{X5}
	\begin{split}
	 \|\frac{\partial }{\partial y}(\frac{\partial p}{\partial x})\|_{ L^{4}_{[0,T_*]} H^{s-1}(\mathbb{R}^2) }  \leq & C\|\bF^{-1}\|^2_{L^{\infty}_{[0,T_*]\times \mathbb{R}^2} } \|Q_0(\bG,\bG)\|_{ L^{1}_{[0,T_*]} H^{s-1}(\mathbb{R}^2) }
	 \\
	 \leq & C\|\bG\|^2_{L^{\infty}_{[0,T_*]}H^{s} } \| \frac{\partial  \bG}{\partial t},\frac{\partial \bG}{\partial y} \|_{ L^{4}_{[0,T_*]}L^\infty }  \|\bG\|_{ L^{\infty}_{[0,T_*]} H^{s} }
	 \\
	 \leq & 27C( \|\bU_0\|_{H^{s+1}}+\|\bv_0\|_{H^{s}})^3 \| \frac{\partial  \bG}{\partial t},\frac{\partial \bG}{\partial y} \|_{ L^{4}_{[0,T_*]}L^\infty } .
	\end{split}
\end{equation}
Inserting \eqref{X5} to \eqref{X4}, we get
\begin{equation}\label{X6}
		 \| d\bG  \|_{ L^{4}_{[0,T_*]}L^\infty (\mathbb{R}^2)} 
		\leq 2C  ( \|\bU_0\|_{H^{s+1}(\mathbb{R}^2)}+\|\bv_0\|_{H^{s}(\mathbb{R}^2)}  ).
\end{equation}
By using \eqref{a07}, we have
\begin{equation}\label{X6a}
	\| d\bF  \|_{ L^{4}_{[0,T_*]}L^\infty (\mathbb{R}^2)} 
	\leq 2C  ( \|\bU_0\|_{H^{s+1}(\mathbb{R}^2)}+\|\bv_0\|_{H^{s}(\mathbb{R}^2)}  ).
\end{equation}
Using energy estimates of \eqref{p0} and H\"older's inequality, we derive that
\begin{equation}\label{X7}
	\begin{split}
		\| \bG\|_{ L^{\infty}_{[0,T_*]}H^s} + \| \frac{\partial  \bG}{\partial t}  \|_{ L^{\infty}_{[0,T_*]}H^{s-1}}
		\leq  & ( \|\bU_0\|_{H^{s+1}}+\|\bv_0\|_{H^{s}} + C\|\frac{\partial }{\partial y}(\frac{\partial p}{\partial x})\|_{ L^{1}_{[0,T_*]} H^{s-1} }  )
		\\
		\leq  & ( \|\bU_0\|_{H^{s+1}}+\|\bv_0\|_{H^{s}} + CT_*^{\frac34}\|\frac{\partial }{\partial y}(\frac{\partial p}{\partial x})\|_{ L^{4}_{[0,T_*]} H^{s-1} }  ).
	\end{split}
\end{equation}
By using \eqref{X5}, \eqref{X6}, \eqref{Tstar}, we conclude that
\begin{equation}\label{X8}
	\begin{split}
	\| \bG\|_{ L^{\infty}_{[0,T_*]}H^s} + \| \frac{\partial  \bG}{\partial t}  \|_{ L^{\infty}_{[0,T_*]}H^{s-1}}
	\leq & 2(\|\bU_0\|_{H^{s+1}}+\|\bv_0\|_{H^{s}}).
	\end{split}
\end{equation}
Using continuous induction argument, then \eqref{X2} holds. In a similar way, the estimate \eqref{X3} also holds. Hence, we have proved the existence of solutions. Furthermore, we can also get the following Strichartz estimates
\begin{equation}\label{X9a}
	\begin{split}
		\| d(\frac{\partial \bU}{\partial y})\|_{ L^{4}_{[0,T_*]}L^\infty} 
		\leq   & C(\|\bU_0\|_{H^{s+1}}+\|\bv_0\|_{H^{s}}).
	\end{split}
\end{equation}
Using \eqref{aa0}, \eqref{in6}, \eqref{in11}, and \eqref{X9a}, we deduce that
\begin{equation}\label{X9aa}
	\begin{split}
		\| d \bv \|_{ L^{4}_{[0,T_*]}L^\infty} 
		\leq   & C(\|\bU_0\|_{H^{s+1}}+\|\bv_0\|_{H^{s}}).
	\end{split}
\end{equation}

\textit{Case 2:}$n=3$. For $s>2$, using Strichartz estimates of linear wave equations (cf. \cite{KeelTao}), it yields
\begin{equation}\label{X9}
	\begin{split}
		\| d\bG  \|_{ L^{2}_{[0,T_*]}L^\infty } 
	\leq & C  ( \|\bU_0\|_{H^{s+1}}+\|\bv_0\|_{H^{s}} + \|\frac{\partial }{\partial y}(\frac{\partial p}{\partial x})\|_{ L^{1}_{[0,T_*]} H^{s-1} }  )	
	\\
	\leq & C  ( \|\bU_0\|_{H^{s+1}}+\|\bv_0\|_{H^{s}} + (T_*)^{\frac12}\|\frac{\partial }{\partial y}(\frac{\partial p}{\partial x})\|_{ L^{2}_{[0,T_*]} H^{s-1} }  ).
	\end{split}
\end{equation}
Noting \eqref{t20}, \eqref{t00}, \eqref{X2}, \eqref{Tstar} and using H\"older's inequality, so we obtain
\begin{equation}\label{X10}
	\begin{split}
		\|\frac{\partial }{\partial y}(\frac{\partial p}{\partial x})\|_{ L^{2}_{[0,T_*]} H^{s-1}(\mathbb{R}^3) }  \leq & C\|\bF^{-1}\|^2_{L^{\infty}_{[0,T_*]\times \mathbb{R}^3} } \|Q_0(\bG,\bG)\|_{ L^{1}_{[0,T_*]} H^{s-1}(\mathbb{R}^3) }
		\\
		\leq & C\|\bG\|^4_{L^{\infty}_{[0,T_*]}H^{s} } \| \frac{\partial  \bG}{\partial t},\frac{\partial \bG}{\partial y} \|_{ L^{4}_{[0,T_*]}L^\infty }  \|\bG\|_{ L^{\infty}_{[0,T_*]} H^{s} }
		\\
		\leq & 243C(\|\bU_0\|_{H^{s+1}}+\|\bv_0\|_{H^{s}})^4 \| \frac{\partial  \bG}{\partial t},\frac{\partial \bG}{\partial y} \|_{ L^{2}_{[0,T_*]}L^\infty } .
	\end{split}
\end{equation}
Inserting \eqref{X10} to \eqref{X9}, we get
\begin{equation}\label{X11}
	\| d\bG  \|_{ L^{2}_{[0,T_*]}L^\infty (\mathbb{R}^3)} 
	\leq 2C  ( \|\bU_0\|_{H^{s+1}(\mathbb{R}^3)}+\|\bv_0\|_{H^{s}(\mathbb{R}^3)}  ).
\end{equation}
By using \eqref{a07}, it yields
\begin{equation}\label{X11a}
	\| d\bF  \|_{ L^{2}_{[0,T_*]}L^\infty (\mathbb{R}^3)} 
	\leq 2C  ( \|\bU_0\|_{H^{s+1}(\mathbb{R}^3)}+\|\bv_0\|_{H^{s}(\mathbb{R}^3)}  ).
\end{equation}
Using energy estimates of \eqref{p0} and H\"older's inequality, we derive that
\begin{equation}\label{X12}
	\begin{split}
		\| \bG\|_{ L^{\infty}_{[0,T_*]}H^s} + \| \frac{\partial  \bG}{\partial t}  \|_{ L^{\infty}_{[0,T_*]}H^{s-1}}
		\leq  & ( \|\bU_0\|_{H^{s+1}}+\|\bv_0\|_{H^{s}} + C\|\frac{\partial }{\partial y}(\frac{\partial p}{\partial x})\|_{ L^{1}_{[0,T_*]} H^{s-1} }  )
		\\
		\leq  & ( \|\bU_0\|_{H^{s+1}}+\|\bv_0\|_{H^{s}} + CT_*^{\frac12}\|\frac{\partial }{\partial y}(\frac{\partial p}{\partial x})\|_{ L^{2}_{[0,T_*]} H^{s-1} }  ).
	\end{split}
\end{equation}
By using \eqref{X5}, \eqref{X6}, \eqref{Tstar}, we conclude that
\begin{equation}\label{X13}
	\begin{split}
		\| \bG\|_{ L^{\infty}_{[0,T_*]}H^s} + \| \frac{\partial  \bG}{\partial t}  \|_{ L^{\infty}_{[0,T_*]}H^{s-1}}
		\leq   & 2(\|\bU_0\|_{H^{s+1}}+\|\bv_0\|_{H^{s}}).
	\end{split}
\end{equation}
Using the continuous induction argument, then \eqref{X2} holds. In a similar way, we can also derive \eqref{X3}. Furthermore, we can get the following Strichartz estimates
\begin{equation}\label{X13a}
	\begin{split}
		\| d(\frac{\partial \bU}{\partial y})\|_{ L^{2}_{[0,T_*]}L^\infty} 
		\leq   & C(\|\bU_0\|_{H^{s+1}}+\|\bv_0\|_{H^{s}}).
	\end{split}
\end{equation}
Using \eqref{aa0}, \eqref{in6}, \eqref{in11}, and \eqref{X13a}, we deduce that
\begin{equation}\label{X13aa}
	\begin{split}
		\| d \bv\|_{ L^{2}_{[0,T_*]}L^\infty} 
		\leq   & C(\|\bU_0\|_{H^{s+1}}+\|\bv_0\|_{H^{s}}).
	\end{split}
\end{equation}
Therefore, for $n=2$ or $3$, the bounds \eqref{X2} and \eqref{X3} both hold. By \eqref{X9aa} and \eqref{X13aa}, we also get \eqref{thm20}. So we have finished the proof of Theorem \ref{thm2}.
\subsection{Proof of Theorem \ref{thmb}}
We will also use the contraction principle to prove the existence and uniqueness of solution. Recall \eqref{p0}, we need some information of $\frac{\partial}{\partial y}\partial_x p$. By chain's rule, so we calculate
\begin{equation}\label{g2}
	\begin{split}
		\Delta_x p=&\frac{\partial}{\partial y^b}  ( \frac{\partial p}{\partial y^a} \frac{\partial y^a}{\partial x^i} ) \frac{\partial y^b}{\partial x_i}
		\\
		=&  \frac{\partial^2 p}{\partial y^a \partial y^b} (\bF^{-1})^{ai}(\bF^{-1})^{b}_{\ i} 
		+ \frac{\partial p}{\partial y^a}  \frac{\partial (\bF^{-1})^{ai}}{\partial y^b} (\bF^{-1})^{b}_{\ i} 
		\\
		=& \Delta_y p+  \frac{\partial^2 p}{\partial y^a \partial y^b} \left\{  (\bF^{-1})^{ai}(\bF^{-1})^{b}_{\ i} -\delta^{ab} \right\}
		+ \frac{\partial p}{\partial y^a}  \frac{\partial (\bF^{-1})^{ai}}{\partial y^b} (\bF^{-1})^{b}_{\ i} .
	\end{split}
\end{equation}
For $\theta > \frac12$ and $s>\frac{n}{2}$, we define the solution space
\begin{equation}\label{Xd}
Y_{s,\theta}=\left\{(\bG,\bv)\in \mathcal{S}'(\mathbb{R}^{1+n}): |\bG|_{s,\theta}\leq C_2\eta \right\},
\end{equation}
where $C_2$ will be defined later. For any $(\bG,\bv) \in Y_{s,\theta}$, we denote the map $\text{M}$ by
\begin{equation}\label{z0}
	\begin{split}
		\text{M} G^{ia}= 	&\chi(t) \left\{ \cos(tD) (\frac{\partial U^i_0}{\partial y^a}) +  D^{-1} \sin(tD) (\frac{\partial v^i_0}{\partial y^a}) \right\} 
		\\
		& + \chi(\frac{t}{T}) \int^t_0 D^{-1} \sin( (t-t')D ) \left\{  \phi( T^{\frac12}\Lambda_{-} ) \frac{\partial }{\partial y^a}(\frac{\partial p}{\partial x_i}) \right\} (t')dt'
		\\
		&+  \square^{-1} ( 1-\phi( T^{\frac12}\Lambda_{-} ) )  \frac{\partial }{\partial y^a}(\frac{\partial p}{\partial x_i}),
	\end{split}	
\end{equation}
Therefore, we have
\begin{equation}\label{z2}
	\begin{split}
	\square 	\text{M} G^{ia}= 	& \frac{\partial }{\partial y^a}(\frac{\partial p}{\partial x_i}) .
	\end{split}	
\end{equation}
Due to \eqref{z0} and Lemma \ref{nonlinearE}, it yields
\begin{equation}\label{z3}
	|\text{M} \bG|_{s,\theta} \leq C( \| \bU_0\|_{H^{s+1}}+\| \bv_0\|_{H^{s}} + 
	\|\frac{\partial }{\partial y}(\partial_{x} p) \|_{s-1,\theta-1} ),
\end{equation}
In the following, let us bound $\|\frac{\partial }{\partial y}(\partial_{x} p) \|_{s-1,\theta-1}$. By $ \Delta_x p=(\bF^{-1})^{kl}(\bF^{-1})^{mj} Q_0(G^{mk}, G^{jl})$ and Lemma \ref{bilinearE}, we can obtain
\begin{equation}\label{V7}
	\begin{split}
		\|\Delta_x p\|_{s-1,\theta-1} \leq & C\|Q_0(\bG,\bG)\|_{s-1,\theta-1}(1+ |\bG|_{s,\theta}+\dots  + |\bG|^{2(n-1)}_{s,\theta})
		\\
		\leq & C(|\bG|^2_{s,\theta}+ |\bG|^3_{s,\theta}+\dots  + |\bG|^{2n}_{s,\theta}).
	\end{split}
\end{equation}
By \eqref{g2} and \eqref{a06}, it follows
\begin{equation}\label{key1}
	\begin{split}
		& \|\Delta_x p\|_{s-1,\theta-1}
		\\
		 \geq & \|\Delta_y p\|_{s-1,\theta-1} 
		- C \| \frac{\partial^2 p}{\partial y^2}\|_{s-1,\theta-1} (  |\bG|_{s,\theta}+\dots  + |\bG|^{n-1}_{s,\theta}  )  -C \| \frac{\partial p}{\partial y} \frac{\partial \bG}{\partial y} \|_{s-1,\theta-1}.
	\end{split}
\end{equation}
On one hand, we can see
\begin{equation}\label{k10}
	\| \frac{\partial^2 p}{\partial y^2}  \|_{s-1,\theta-1} \approx \|\Delta_y p\|_{s-1,\theta-1}
\end{equation}
On the other hand, from Lemma \ref{prodE}, we can infer that
\begin{equation}\label{k11}
	\begin{split}
			\| \frac{\partial p}{\partial y} \frac{\partial \bG}{\partial y} \|_{s-1,\theta-1}
		\lesssim & \| \frac{\partial p}{\partial y}  \|_{s,\theta-1}  \|  \frac{\partial \bG}{\partial y} \|_{s-1,\theta}
		\\
		\lesssim & ( \| \frac{\partial p}{\partial y}  \|_{0,\theta-1}+\| \frac{\partial^2 p}{\partial y^2}  \|_{s-1,\theta+\epsilon-1}) |\bG|_{s,\theta}
		\\
		\lesssim &  \| \frac{\partial p}{\partial y}  \|_{0,\theta-1} |\bG|_{s,\theta} +\| \frac{\partial^2 p}{\partial y^2}  \|_{s-1,\theta-1} |\bG|_{s,\theta},
	\end{split}
\end{equation}
holds. In \eqref{k11}, there is also a lower-order term $\| \frac{\partial p}{\partial y}  \|_{0,\theta-1}$. Noting $\theta-1\leq 0$, then we have
\begin{equation}\label{pd0}
	\begin{split}
			\| \frac{\partial p}{\partial y}  \|_{0,\theta-1}=& \| \frac{\partial p}{\partial x}  \frac{\partial x}{\partial y} \|_{0,\theta-1}
			\\
			\leq & C\| \frac{\partial p}{\partial x} \|_{0,\theta-1} ( 1+ |\bG|_{s,\theta} )
			\\
			 \leq & \| \frac{\partial {p}}{\partial x}  \|_{L^2_{[0,1]} L^2(\mathbb{R}^n_y)}( 1+ |\bG|_{s,\theta} ).
	\end{split}
\end{equation}
Noting \eqref{in13}, using $\dot{B}^0_{2,1} \hookrightarrow L^2$ (\cite{BCD}, Proposition 2.29) and interpolation formula (\cite{BCD}, Corollary 2.55), we can see
\begin{equation}\label{pd1}
	\begin{split}
		\| \frac{\partial {p}}{\partial x}  \|_{ L^2(\mathbb{R}^n_y)}
		=& \| \frac{\partial \bar{p}}{\partial x}  \|_{ L^2(\mathbb{R}^n_x)}
		\\
		= & \| (-\Delta_x)^{-\frac12}  \Delta_x \bar{p} \|_{ L^2(\mathbb{R}^n_x)}
		\\
		\leq & C \| \Delta_x \bar{p} \|_{ \dot{B}^{-1}_{2,1}(\mathbb{R}^n_x)}
		\\
		\leq & C ( \| \partial_x \bar{\bv} \|_{ \dot{H}^{\frac{n}{2}-1}(\mathbb{R}^n_x)} \cdot \|\partial_x \bar{\bv} \|_{ \dot{H}^{0}(\mathbb{R}^n_x)}+ \| \partial_x \bar{\bF} \|_{ \dot{H}^{\frac{n}{2}-1}(\mathbb{R}^n_x)} \cdot \| \partial_x \bar{\bF} \|_{ \dot{H}^{0}(\mathbb{R}^n_x)} )
		\\
			\leq & C ( \| \partial_x \bar{\bv} \|_{ \dot{H}^{\frac{n}{2}-1}(\mathbb{R}^n_y)}  \|\partial_x \bar{\bv} \|_{ \dot{H}^{0}(\mathbb{R}^n_y)}+ \| \partial_x \bar{\bF} \|_{ \dot{H}^{\frac{n}{2}-1}(\mathbb{R}^n_y)}  \| \partial_x \bar{\bF} \|_{ \dot{H}^{0}(\mathbb{R}^n_y)} )\|\frac{\partial \bx}{\partial \by}\|^{n-1}_{L^\infty(\mathbb{R}^n_y)}
		\\
		\leq & C ( \|  \frac{\partial {\bv}}{\partial y} \|_{ \dot{H}^{\frac{n}{2}-1}(\mathbb{R}^n_y)} \|\frac{\partial {\bv}}{\partial y} \|_{ \dot{H}^{0}(\mathbb{R}^n_y)}+ \| \frac{\partial {\bF}}{\partial y}\|_{ \dot{H}^{\frac{n}{2}-1}(\mathbb{R}^n_y)} \| \frac{\partial {\bF}}{\partial y} \|_{ \dot{H}^{0}(\mathbb{R}^n_y)} ) \|\frac{\partial \bx}{\partial \by}\|^{n-1}_{L^\infty(\mathbb{R}^n_y)}
		\\
		\leq & C |\bG|^2_{s,\theta} (1+ |\bG|_{s,\theta}).
	\end{split}
\end{equation}
Above, we use $\frac{\partial {\bv}}{\partial y} =\frac{\partial \bG}{\partial t}$ and $s>\frac{n}{2}$. Substituting \eqref{pd1} to \eqref{pd0}, it yields
\begin{equation}\label{pd2}
		\| \frac{\partial p}{\partial y}  \|_{0,\theta-1}
		\leq  |\bG|^2_{s,\theta}( 1+ |\bG|^2_{s,\theta} ).
\end{equation}
To conclude our outcome \eqref{key1}, \eqref{k10}, \eqref{k11}, and \eqref{pd2}, we have 
\begin{equation}\label{key3}
		 \|\Delta_x p\|_{s-1,\theta-1}
		\geq  \|\Delta_y p\|_{s-1,\theta-1} 
		\left\{ 1- C  ( |\bG|_{s,\theta}+\dots  + |\bG|^{n-1}_{s,\theta}  ) \right\} -C |\bG|^3_{s,\theta} (1+|\bG|^2_{s,\theta}).
\end{equation}
For $\bG \in Y_{s,\theta}$ and $\eta$ is small enough, so we have
\begin{equation}\label{key2}
	1- C  ( |\bG|_{s,\theta}+\dots  + |\bG|^{n-1}_{s,\theta}  ) \geq \frac{1}{2}.
\end{equation}
By \eqref{key3} and \eqref{key2}, we therefore get
\begin{equation}\label{V8}
	\begin{split}
	\|\Delta_y p\|_{s-1,\theta-1} \leq	& C\|\Delta_x p\|_{s-1,\theta-1}
	+ C|\bG|^3_{s,\theta} (1+|\bG|^2_{s,\theta}).
	\end{split}
\end{equation}
Substituting \eqref{V7} to \eqref{V8}, it implies
\begin{equation}\label{V9}
	\begin{split}
		\|\Delta_y p\|_{s-1,\theta-1} \leq	& 
		C(|\bG|^2_{s,\theta}+\dots+|\bG|^{2n}_{s,\theta})\leq C(C_2\eta)^2.
	\end{split}
\end{equation}
Since $ \frac{\partial}{\partial y} ( \partial_x p )=\frac{\partial^2 p}{\partial y^2} \bF^{-1}+ \frac{\partial p}{\partial y} \frac{\partial \bF^{-1}}{\partial y} $, it yields
\begin{equation}\label{V10}
	\begin{split}
		\|\frac{\partial}{\partial y} ( \partial_x p )\|_{s-1,\theta-1} \lesssim	&   \|\frac{\partial^2 p}{\partial y^2} \|_{s-1,\theta-1}( 1+ |\bG|_{s,\theta}+\dots + |\bG|^{n-1}_{s,\theta} ) + \|\frac{\partial p}{\partial y} \|_{s,\theta-1} |\bG|_{s,\theta} .
	\end{split}
\end{equation}
Adding \eqref{pd2} and \eqref{V9} to \eqref{V10}, we have
\begin{equation*}\label{V11}
	\begin{split}
		 \|\frac{\partial}{\partial y} ( \partial_x p )\|_{s-1,\theta-1} 
		\leq	
		& C( |\bG|^2_{s,\theta}+|\bG|^3_{s,\theta}+\cdots |\bG|^{3n-1}_{s,\theta}).
	\end{split}
\end{equation*}
Therefore, we obtain
\begin{equation*}
	\begin{split}
		(1- C|\bG|_{s,\theta} )\|\frac{\partial}{\partial y} ( \partial_x p )\|_{s-1,\theta-1} 
		\leq	
		&  C( |\bG|^2_{s,\theta}+|\bG|^3_{s,\theta}+\cdots |\bG|^{3n-1}_{s,\theta}).
	\end{split}
\end{equation*}
For $\epsilon$ is small enough, so we conclude that
\begin{equation}\label{V12}
	\begin{split}
\|\frac{\partial}{\partial y} ( \partial_x p )\|_{s-1,\theta-1} 
		\leq	
		&  C( |\bG|^2_{s,\theta}+|\bG|^3_{s,\theta}+\cdots |\bG|^{3n-1}_{s,\theta})
		\\
		\leq & C (C_2 \eta)^2.
	\end{split}
\end{equation}
Taking
\begin{equation}\label{T}
	C_2=4(1+C),\quad  0<  \eta \leq \frac{1}{16(1+C)^2},
\end{equation}
using \eqref{z3} and \eqref{V12}, we therefore prove
\begin{equation*}
	|\textrm{M} \bG|_{s,\theta} \leq 2(1+C)\eta.
\end{equation*}
This implies that $\textrm{M}$ is a map from $Y_{s,\theta}$ to $Y_{s,\theta}$. For $\bG, \bG_*\in Y_{s,\theta}$, we can get
\begin{align}\label{V15}
	|\text{M} \bG- \text{M} \bG_*|_{s,\theta} \leq & C
	\|\frac{\partial }{\partial y}(\partial_{x} p-\partial_{x} p_* ) \|_{s-1,\theta-1}.
\end{align}
For $p-p_*$, on one hand, it satisfies
\begin{equation}\label{A10}
	\Delta_x p- \Delta_x p_*=(\bF^{-1})^{kl}(\bF^{-1})^{mj} Q_0(G^{mk}, G^{jl})-(\bF_*^{-1})^{kl}(\bF_*^{-1})^{mj} Q_0(G_*^{mk}, G_*^{jl}).
\end{equation}
Therefore, we can bound it by
\begin{equation}\label{V70}
	\| \Delta_x p- \Delta_x p_* \|_{s-1,\theta+\varepsilon-1} \leq C |\bG-\bG_*|_{s,\theta} ( |\bG|_{s,\theta}+ |\bG_*|_{s,\theta}).
\end{equation}
On the other hand, similar with \eqref{g2}, it yields
\begin{equation}\label{V17}
	\begin{split}
		\Delta_x (p-p_*)
		=& \Delta_y (p-p_*)+  \frac{\partial^2 p}{\partial y^a \partial y^b} \left\{  (\bF^{-1})^{ai}(\bF^{-1})^{b}_{\ i} -\delta^{ab} \right\}-\frac{\partial^2 p_*}{\partial y^a \partial y^b} \left\{  (\bF_*^{-1})^{ai}(\bF_*^{-1})^{b}_{\ i} -\delta^{ab} \right\}
		\\
		&
		+ \frac{\partial p}{\partial y^a}  \frac{\partial (\bF^{-1})^{ai}}{\partial y^b} (\bF^{-1})^{b}_{\ i}
		-\frac{\partial p_*}{\partial y^a}  \frac{\partial (\bF_*^{-1})^{ai}}{\partial y^b} (\bF_*^{-1})^{b}_{\ i} .
	\end{split}
\end{equation}
Similarly, by \eqref{k11}, we can obtain
\begin{align}\label{V30}
	& \| \frac{\partial (p-p_*)}{\partial y} \frac{\partial \bG}{\partial y} \|_{s-1,\theta-1}
	\lesssim  \| \partial_{y} (p-p_*) \|_{0,\theta-1}  |\bG|_{s,\theta}   + \| \frac{\partial^2 (p-p_*)}{\partial y^2} \|_{s-1,\theta-1} |\bG|_{s,\theta},
	\\
	\label{V40}
	& \| \frac{\partial p_*}{\partial y} \frac{\partial (\bG-\bG_*) }{\partial y} \|_{s-1,\theta-1}\lesssim  \| \partial_{y} p_* \|_{0,\theta-1}  |\bG-\bG_*|_{s,\theta}   + \| \frac{\partial^2 p_*}{\partial y^2} \|_{s-1,\theta-1} |\bG-\bG_*|_{s,\theta}.
\end{align}
Similarly, by \eqref{pd2}, we can obtain
\begin{equation}\label{V50}
	\begin{split}
		\| \partial_{y} (p-p_*) \|_{0,\theta-1}  \lesssim  & |\bG-\bG_*|_{s,\theta} ( |\bG|_{s,\theta}+|\bG_*|_{s,\theta} ).
	\\
	\| \partial_{y} p_* \|_{0,\theta-1} \lesssim  & |\bG_*|^2_{s,\theta}.
	\end{split}
\end{equation}
Due to \eqref{V17}, \eqref{V30}, \eqref{V40}, and \eqref{V50}, we have
\begin{equation*}
	\begin{split}
		& \|\Delta_x (p-p_*)\|_{s-1,\theta-1}
		\\
		\geq & \|\Delta_y (p-p_*)\|_{s-1,\theta-1} 
		- C \| \frac{\partial^2 (p-p_*)}{\partial y^2}\|_{s-1,\theta-1} (  |\bG|_{s,\theta}  + |\bG|^{n-1}_{s,\theta} + |\bG_*|_{s,\theta}  + |\bG_*|^{n-1}_{s,\theta})  
		\\
		& - C \| \frac{\partial^2 p_*}{\partial y^2}\|_{s-1,\theta-1}   |\bG-\bG_*|_{s,\theta} (|\bG|_{s,\theta}   + |\bG|^{n-2}_{s,\theta} + |\bG_*|_{s,\theta}  + |\bG_*|^{n-2}_{s,\theta})  
		\\
		& -C \| \frac{\partial (p-p_*)}{\partial y} \frac{\partial \bG}{\partial y} \|_{s-1,\theta-1}  -C \| \frac{\partial p_*}{\partial y} \frac{\partial (\bG-\bG_*) }{\partial y} \|_{s-1,\theta-1}
		\\
		\geq & \frac12\|\Delta_y (p-p_*)\|_{s-1,\theta-1} - C  |\bG-\bG_*|_{s,\theta}  (|\bG|_{s,\theta}+|\bG_*|_{s,\theta})
	\end{split}
\end{equation*}
As a result, we get
\begin{equation}\label{V80}
		\|\Delta_y (p-p_*)\|_{s-1,\theta-1} \leq  C\|\Delta_x (p-p_*)\|_{s-1,\theta-1}
	+ C   |\bG-\bG_*|_{s,\theta}  (|\bG_{s,\theta}+|\bG_*|_{s,\theta}).
\end{equation}
Inserting \eqref{V70} to \eqref{V80}, it tells us
\begin{equation}\label{V90}
	\|\Delta_y (p-p_*)\|_{s-1,\theta-1} \leq   C   |\bG-\bG_*|_{s,\theta}  (|\bG_{s,\theta}+|\bG_*|_{s,\theta}).
\end{equation}
Since
\begin{equation*}
	\frac{\partial}{\partial y} ( \partial_x p-\partial_x p_* )=\frac{\partial^2 p}{\partial y^2} \bF^{-1}+ \frac{\partial p}{\partial y} \frac{\partial \bF^{-1}}{\partial y}-( \frac{\partial^2 p_*}{\partial y^2} \bF_*^{-1} + \frac{\partial p_*}{\partial y} \frac{\partial \bF_*^{-1}}{\partial y}),
\end{equation*}
we have
\begin{equation}\label{V100}
	\begin{split}
		& \|\frac{\partial}{\partial y} ( \partial_x p-\partial_x  p_* )\|_{s-1,\theta-1} 
		\\
		\lesssim	&   \|\frac{\partial^2 (p-p_*)}{\partial y^2} \|_{s-1,\theta-1}( 1+ |\bG|_{s,\theta}+\dots + |\bG|^{n-1}_{s,\theta} ) 
		\\
		&+ \|\frac{\partial^2 p_*}{\partial y^2} \|_{s-1,\theta-1} |\bG-\bG_*|_{s,\theta} ( |\bG|_{s,\theta}+ |\bG_*|_{s,\theta}+|\bG|_{s,\theta}+ |\bG_*|^{n-2}_{s,\theta}  )
		\\
		& + \|\frac{\partial (p-p_*)}{\partial y} \|_{s,\theta-1} |\bG|_{s,\theta}+ \|\frac{\partial p_*}{\partial y} \|_{s,\theta-1} |\bG-\bG_*|_{s,\theta} .
	\end{split}
\end{equation}
Combing with \eqref{V90} and \eqref{V100}, we conclude that
\begin{equation}\label{A20}
	\|\frac{\partial}{\partial y} ( \partial_x p-\partial_x  p_* )\|_{s-1,\theta-1}  \leq C   |\bG-\bG_*|_{s,\theta}  (|\bG_{s,\theta}+|\bG_*|_{s,\theta}).
\end{equation}
By \eqref{V15} and \eqref{A20}, it yields
\begin{align*}
	|\text{M} \bG- \text{M} \bG_*|_{s,\theta}  \leq  
	C   |\bG-\bG_*|_{s,\theta}  (|\bG_{s,\theta}+|\bG_*|_{s,\theta}) .
\end{align*}
Noting \eqref{T}, so we have
\begin{align}\label{A21}
	|\text{M} \bG- \text{M} \bG_*|_{s,\theta} \leq & \frac12
	|\bG-\bG_*|_{s,\theta}.
\end{align}
Therefore, the map $\textrm{M}$ is a contraction. By the contraction mapping theorem, there is a unique solution to \eqref{p0}. Moreover, the solution satisfies
\begin{equation*}
	| \bv |_{s,\theta}+ | \bF |_{s,\theta} \lesssim \eta.
\end{equation*}
In a similar way, we can prove the continuous dependence of solution. At this stage, we have finished the proof of Theorem \ref{thmb}.
\section*{Acknowledgments} 
The authors sincerely express a great attitude to the reviewers for their helpful comments. 
The author is supported by National Natural Science Foundation of China (Grant No.12101079) and the Fundamental Research Funds for the Central Universities (Grant No.531118010867).

\section*{Conflicts of interest and Data Availability Statements}
The authors declared that this work does not have any conflicts of interest. The authors also confirm that the data supporting the findings of this study are available within the article.

\end{document}